\documentclass[11pt]{article}

\usepackage{amsmath,amssymb,amsfonts,amsthm}
\usepackage{moreverb,rotating,graphics}
\usepackage[latin1]{inputenc}
\usepackage{epsfig}
\usepackage[english]{babel} 
\usepackage{color}
\usepackage[nottoc, notlof, notlot]{tocbibind}
\usepackage{psfrag}

\newtheorem{theorem}{Theorem}[section]
\newtheorem{proposition}[theorem]{Proposition}
\newtheorem{remark}[theorem]{Remark}
\newtheorem{lemma}[theorem]{Lemma}

\def\P{{\mathcal P}}

\def\F{{\mathcal F}}

\def\O{{\Omega}}

\def\NN{{\mathbb N}}
\def\ZZ{{\mathbb Z}}
\def\RR{{\mathbb R}}

\def\PP{{\mathbb P}}
\def\EE{{\mathbb E}}
\def\1{{ 1 }}

\newcommand{\set}[1]{\left\{ #1\right\}}
\newcommand{\bra}[1]{\left( #1\right)}

\renewcommand{\phi}{\varphi}
\newcommand \dps{\displaystyle }

\renewcommand{\phi}{\varphi}
\newcommand{\Id}{\text{Id}}

\title{Effective dynamics for a kinetic Monte-Carlo model with slow and
  fast time scales}
\author{Salma Lahbabi$^{1,2,3}$ and Frédéric Legoll$^{4,2}$\\
{\footnotesize $^1$ CERMICS, \'Ecole des Ponts ParisTech, Universit\'e Paris-Est,}\\
{\footnotesize 6 et 8 avenue Blaise Pascal, 77455 Marne-La-Vall\'ee
Cedex 2, France}
\\
{\footnotesize \tt lahbabis@cermics.enpc.fr}
\\
{\footnotesize $^2$ INRIA Rocquencourt, MICMAC team-project,}\\
 {\footnotesize Domaine de Voluceau, B.P. 105,
 78153 Le Chesnay Cedex, France}
\\
{\footnotesize $^3$ CNRS \& Laboratoire de Mathématiques (UMR 8088), Université de Cergy-Pontoise,}\\
{\footnotesize 95000 Cergy-Pontoise Cedex, France}
\\
{\footnotesize $^4$ Laboratoire Navier, \'Ecole des Ponts ParisTech, Universit\'e Paris-Est,}\\
{\footnotesize 6 et 8 avenue Blaise Pascal, 77455 Marne-La-Vall\'ee
Cedex 2, France}\\
{\footnotesize \tt legoll@lami.enpc.fr}
}
\date{\today} 

\begin{document}

\maketitle

\begin{abstract}
We consider several multiscale-in-time kinetic Monte Carlo models, in which 
some variables evolve on a fast time scale, while the others
evolve on a slow time scale. In the first two models we consider,
a particle evolves in a one-dimensional potential energy landscape which
has some 
small and some large barriers, the latter dividing the state space into
metastable regions. In the limit of infinitely large barriers, we
identify the effective dynamics 
between these macro-states, and prove the convergence of the
process towards a kinetic Monte Carlo model. We next
consider a third model, which consists of a system of two particles. The
state of each particle 
evolves on a fast time-scale while conserving their respective
energy. In addition, the particles can exchange
energy on a slow time scale. Considering the energy of the first
particle, we identify its effective
dynamics in the limit of asymptotically small ratio between
the characteristic times of the fast and the slow dynamics. For all
models, our 
results are illustrated by representative numerical simulations. 
\end{abstract}



\section{Introduction}



Langevin dynamics is commonly used in computational
statistical physics to model the evolution of atomistic systems at finite
temperature. The state of the system evolves according to a stochastic
differential equation, and is thus modelled as a real vector valued Markov
process. Generically, the state space of such atomistic systems can be decomposed into several
metastable regions, separated by high energy barriers. It is therefore
natural to introduce kinetic Monte-Carlo models as a simplification of the
continuous-in-space reference model, where the state space is
coarse-grained into discrete states that each corresponds to a
metastable region of the continuous model. We refer
e.g. to~\cite{par_rep} for a formalization of this idea. The resulting
dynamics is a time continuous Markov chain, also called jump process.

In this work, we consider such a jump process, with the particularity
that two different time scales are present in the system. 
On a typical trajectory, many jumps of the fast degrees of freedom occur
before a significant evolution of the slowly varying variables is
observed. Therefore, a direct discretization is numerically very
costly (this problem is known as the small barrier problem).  
The aim of this work is to find an effective
dynamics for the slow variables (which turns out to be again a
kinetic Monte Carlo model) while filtering out the fast variables. This
effective dynamics is derived in the regime of large time scale
separation between the slow and the fast variables. 

\medskip

We will successively perform this derivation for three different models.

\medskip

First, in Section~\ref{modele_deux_etats}, we consider a particle
subjected to a potential energy presenting two macro-states separated by
a high energy barrier. Inside each macro-state, there are finitely many
micro-states separated by relatively low energy barriers (see
Fig.~\ref{fig:puits}). The ratio between the low energy barriers and
the large energy barriers is characterized by a parameter $\epsilon$
that we will take asymptotically small. This ratio encodes the
difference of time scales between the dynamics within a macro-state
(only low energy barriers have to be overcome, and the dynamics is
therefore fast), and the global dynamics (for which large energy
barriers have to be overcome, making this dynamics slow). See
Section~\ref{section_1_1} for a complete description of the model.

We are interested in the long time behavior of
functions of the slow variables. We consider in this article the
simplest case of such function, that is, the macro-state in which the
particle is 
located. At the price of additional technicalities, our approach carries
over to more general functions of the slow variables. 

Under an irreducibility assumption on the dynamics within the 
macro-states, we prove that, in the limit of asymptotically large time
scale separation (namely when $\epsilon$ goes to zero), the dynamics of
the slow variable converges to a jump process over the two
macro-states. The transition rates of this limiting process are, in some
sense, the weighted averages of the transition rates of the reference
model. We underline that our convergence is a convergence on the {\em
  path} of the system, and not only on the state of the system at any
given time. Our main result, Theorem~\ref{th1}, is presented in
Section~\ref{section_1_1} and proved in Section~\ref{section_1_2}. 

In Section~\ref{1}, we present detailed numerical results illustrating
our theoretical conclusions. In particular, we monitor the probability
distribution of the first waiting time in a macro-state, and check that
this distribution indeed converges to the asymptotic distribution. 

\medskip

In Section~\ref{infinite_detats}, we turn to our second model, which is
a generalization of the model considered in
Section~\ref{modele_deux_etats} where the potential energy presents {\em
  infinitely many} macro-states instead of two. To simplify the problem,
we assume that the internal dynamics within each macro-state are
identical (see Section~\ref{sym} for a detailed presentation of the
model). In this case, the effective dynamics is a time continuous random
walk with Poissonian waiting times, as stated in our main
result of that Section, Theorem~\ref{th2}. We provide some
representative numerical 
results in Section~\ref{2}. 

\medskip

We finally turn in Section~\ref{sec_ene} to our third model, which is
different in spirit from the models studied in
Sections~\ref{modele_deux_etats} and~\ref{infinite_detats}. One interest
of this last section is to show that the arguments employed to analyze
the first two models can be used to study a model different in
nature. The system at hand in Section~\ref{sec_ene} contains two
particles, each one being described by $k$ spin-like variables. The system
evolves either due to the internal evolution of each particle (which
occurs on a fast time-scale), or due to the interaction between the two
particles (which occurs on a slow time-scale). In the first case, the
energy of each particle is preserved 
while in the second, there is an exchange of energy between the two
particles. Note that the total energy of the system is preserved in both
cases. Our
quantity of interest is the energy of the first 
particle, which is indeed a slow observable (see Section~\ref{pres_ene}
for a complete description of the model). We show that the dynamics
of the first particle energy converges to a jump process
on the (finite) set of admissible energies, this set being determined by the
initial energy (see Section~\ref{section_3_2}, Theorem~\ref{energie}, for
our main result). We collect in Section~\ref{sec_sim_ene} some
numerical illustrations. 

\bigskip

The difficulty of the question we address stems from the fact
that the slow observable is not a Markov process: this is a closure
problem. A typical tool in this context is the Mori-Zwanzig
projection formalism, which is described in details
in~\cite{gks}. This leads to approximating the slow observable by a
process which has some memory in time. In our work, we assume that a
time-scale separation is present in the system. Memory effects may then
be neglected, and the slow observables be approximated by a Markov
process. As often the case in such settings, an essential ingredient of
our proof is an averaging principle (see~\cite{pav_stu} for a
comprehensive review of that principle in various contexts). 
We refer to~\cite{friesecke,novotny,schuette_jcp,schuette_handbook} for
related works in the framework of discrete time Markov chains in a
discrete state space. 

As pointed out above, kinetic Monte Carlo models are somewhat obtained
as a coarse-grained approximation of real valued Markov processes, such
as the Langevin equation (or its overdamped limit). In that framework,
the construction and the analysis of effective dynamics has been
undertaken in several works, see e.g.~\cite{eff_dyn,eff_dyn_lncse} and
the comprehensive bibliography contained therein.

\medskip

Throughout this paper, we use several well-known results that we recall
in Appendix~\ref{sec:app} below. 

\section{A particle in a potential energy landscape with two macro-states}
\label{modele_deux_etats}

In this section we study the dynamics of a particle in a potential energy
with two macro-states (see Fig.~\ref{fig:puits}). The state of the
particle is represented by a macroscopic variable (the index of the
macro-state), which can take here only two values, and a microscopic
variable (the index of the micro-state within the macro-state). We are
concerned with the long time behaviour of the macroscopic variable. In
Section~\ref{section_1_1}, we present the model and state our convergence
result (Theorem~\ref{th1}), the proof of which is given in
Section~\ref{section_1_2}. Numerical results illustrating our theoretical
conclusions are gathered in Section~\ref{1}.

\setlength{\unitlength}{0.25cm}
\linethickness{0.1mm}
\begin{figure}[h]
\begin{center}
\begin{picture}(30,10)
\qbezier(5,9.5)(6,10)(7,0)
\qbezier(7,0)(7.5,-1)(8,0)
\qbezier(8,0)(8.5,1)(9,0)
\qbezier(9,0)(9.5,-1)(10,0)
\qbezier(10,0)(10.5,1)(11,0)
\qbezier(11,0)(11.5,-1)(12,0)
\qbezier(12,0)(12.5,1)(13,0)
\qbezier(13,0)(13.5,-1)(14,0)
\qbezier(14,0)(15,20)(16,0)

\qbezier(16,0)(16.5,-1)(17,0)
\qbezier(17,0)(17.5,1)(18,0)
\qbezier(18,0)(18.5,-1)(19,0)
\qbezier(19,0)(19.5,1)(20,0)
\qbezier(20,0)(20.5,-1)(21,0)
\qbezier(21,0)(21.5,1)(22,0)
\qbezier(22,0)(22.5,-1)(23,0)
\qbezier(23,0)(24,10)(25,9.5)
\end{picture}
\end{center}
\caption{Example of a potential energy with two macro-states of energy wells.
\label{fig:puits}}
\end{figure}
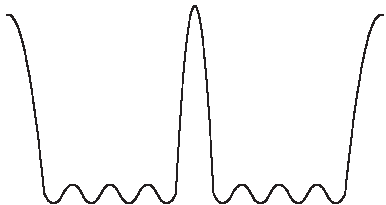

\subsection{Presentation of the model and main result}
\label{section_1_1}

We now formalize the model described above. We introduce a parameter
$\epsilon$ which represents the ratio between the characteristic time of
the internal dynamic inside a given macro-state (fast time scale) and
the characteristic time of evolution of the macro-state, namely the
characteristic time the system spends in a given macro-state before
going to the other one. For simplicity, we assume that both macro-states
contain the 
same number of micro-states. The macro-states are labelled by 0 and 1,
whereas the micro-states are labelled as 1, 2, \dots, $m$. We set
$M=\left\{ 1, 2, \dots, m \right\}$. 

The state of the particle is modelled by a time continuous Markov chain
$\overline{Y_t^{\epsilon}} = \left( \overline{X_t^{\epsilon}}, 
\overline{Z_t^{\epsilon}}\right)$,
which takes its values in the space $E=M\times \{0,1\}$. The first
coordinate of $\overline{Y_t^{\epsilon}}$ represents the micro-state of
the particle inside a given macro-state, and thus takes its value in $M$. The
second coordinate determines in which 
macro-state the particle is located at time $t$:
$\overline{Z_t^{\epsilon}}=0$ or $1$.

We denote by $\overline{Q}^{\epsilon}$ the transition matrix of the
process $\overline{Y_t^{\epsilon}}$. Let $Q_0$ and $Q_1$ be two $m\times
m$ matrices 
that determine the internal dynamic within each macro-state and let
$C_{0,1}$ and $C_{1,0}$ be two 
$m\times m$ matrices that determine the coupling between micro-states that
belong to different macro-states. The transition rates of
$\overline{Y_t^{\epsilon}}$ are given by 
\begin{eqnarray*}
\overline{Q}^{\epsilon}\left( \left(x,z\right), \left(x',z\right)
\right)
&=& 
Q_z\left(x,x'\right), 
\quad \text{$z=0$ or 1, $x \neq x'$},
\\
\overline{Q}^{\epsilon}\left( \left(x,z\right), \left(x',z'\right)
\right)
&=&
\epsilon C_{z,1-z}\left(x,x'\right) \quad \text{for $z\neq z'$}.
\end{eqnarray*}
Thus, $\overline{Q}^\epsilon$ is of the form
$$
\overline{Q}^\epsilon = \left(
\begin{array}{cc}
Q_0 &\epsilon C_{0,1} \\
\epsilon C_{1,0}  & Q_1
\end{array}
\right).
$$
\begin{remark}
\label{rem:convention}
As always for Markov jump processes, the diagonal entries of the
transition matrix are irrelevant. Our convention is to take them equal
to zero. 
\end{remark}

The process $\overline{Y_t^{\epsilon}}$ is a jump process. It means
that, when it is in a state $\left(x,z\right)$, then
\begin{itemize}
\item it stays there for a time $S$, which is a random variable
  distributed according to an exponential distribution of parameter 
$$
\overline{q}^{\epsilon}\left(x,z\right)
:=
\sum_{\left(x',z'\right)\in E \atop 
\left(x',z'\right) \neq \left(x,z\right)}
\overline{Q}^{\epsilon}
\left(\left(x,z\right),\left(x',z'\right)\right),
$$ 
that is $\PP\left(  S\leq t \right)=1-\exp \left(
  -\overline{q}^{\epsilon}\left(x,z\right)t \right)$.
\item At this time $S$, it jumps to another state. The probability
  that it jumps to the state $\left(x',z'\right)\neq (x,z)$ is given by 
$$
\frac{\overline{Q}^{\epsilon}\left(\left(x,z\right),\left(x',z'\right)\right)}{\overline{q}^{\epsilon}\left(x,z\right)}.
$$ 
\end{itemize}
Note that the paths of a jump process are by convention right
continuous, with left limits (they are thus c\`ad-l\`ag functions). 

 
\medskip

We are interested in the behaviour of a macroscopic observable, that is
a function of the slow variable $\overline{Z}_t^{\epsilon}$. The dynamic
inside a given macro-state, i.e. when the variable $z$ does not change,
has a characteristic time of the order of $O\left(1\right)$
(i.e. independent of $\epsilon$), whereas the characteristic time for
the particle to go from one macro-state to the other is of the order of
$O\left(\epsilon^{-1}\right)$. We therefore consider henceforth
the rescaled-in-time process 
$\left(\overline{Z}_{t/\epsilon}^{\epsilon}\right)_t$. We introduce the
process $Y^\epsilon_t:=\overline{Y}_{t/\epsilon}^{\epsilon}$, which is a
jump process of intensity matrix $Q^\epsilon$ given by 
\begin{equation}
\label{eq:def_Q_eps}
Q^\epsilon = \left(
\begin{array}{cc}
\epsilon^{-1} Q_0 & C_{0,1}  \\
\ C_{1,0}  & \epsilon^{-1} Q_1
\end{array}
\right).
\end{equation}
We assume that 
\begin{equation}
\label{eq:hyp}
\text{the matrices $Q_0$ and $Q_1$ are irreducible},
\end{equation} 
therefore admitting unique invariant measures denoted by $\pi_0$ and $\pi_1$,
respectively. 

\begin{remark}
\label{rem:convention2}
Due to our convention on the transition matrix (see
Remark~\ref{rem:convention}), the invariant measure $\pi$ of a
transition matrix $Q = \left\{ q_{i,j} \right\}_{1 \leq i,j \leq m}$
satisfies $\pi^T Q= \pi^T \Delta$,
where $\Delta$ is a diagonal matrix with $\Delta_i = \sum_{j=1}^m q_{i,j}$. 
\end{remark}

\paragraph{Definitions and notations}

We denote by $D_{\mathbb R}\left[0,\infty\right)$ the set of c\`ad-l\`ag
functions defined on $\left[0,\infty\right)$ and valued in $\mathbb
R$, and by
$C_{\mathbb R}\left[0,\infty\right)$ the set of continuous functions
defined on $\left[0,\infty\right)$ and valued in $\mathbb R$. Endowed
with the Skorohod metric (see e.g.~\cite[p.~116--118]{kurtz}),
$D_{\mathbb R}\left[0,\infty\right)$ is a complete separable space. 

A family of probability measures $\P_n$ on $D_{\mathbb
  R}\left[0,\infty\right)$ is said to {\it weakly} converge to a
probability measure $\P$ on $D_{\mathbb R}\left[0,\infty\right)$ if, for
any bounded continuous function $\Phi$ on $D_{\mathbb R}
\left[0,\infty\right)$,
$$
\lim_{n\rightarrow \infty} \int \Phi \,d\P_n=\int \Phi\,d\P.
$$
A family of random variables $X_n$ valued in $D_{\mathbb
  R}\left[0,\infty\right)$ is said to converge {\it in distribution} to
$X \in D_{\mathbb R}\left[0,\infty\right)$ if the distribution of $X_n$
weakly converges to the distribution of $X$. Otherwise stated, the
family $X_n$ converges in distribution to $X$ if, for any bounded
continuous function $\Phi$ on $D_{\mathbb R}\left[0,\infty\right)$, we
have
$$
\lim_{n\rightarrow \infty} \EE \left[ \Phi(X_n) \right] = 
\EE \left[ \Phi(X) \right].
$$
Throughout this article, we use the symbol $\Rightarrow$ for the
convergence in distribution of c\`ad-l\`ag stochastic processes or the weak
convergence of their corresponding distributions.

\paragraph{Main result}

We are now in position to present the main result of this section. 
For $z\in\set{0,1}$, we define
$$
\overline{C}_{z,1-z}\left(x\right) = \sum_{x'\in M} C_{z,1-z}\left(x,x'\right)
$$ 
and
\begin{equation}
\label{eq:def_lambda}
\lambda_z
=
\sum_{x\in M} \overline{C}_{z,1-z}\left(x\right) \pi_z \left(x\right)
=
\sum_{x\in M} \pi_z \left(x\right) \sum_{x'\in M} C_{z,1-z}\left(x,x'\right).
\end{equation} 

\begin{theorem}
\label{th1}
Let $Y^\epsilon_t=(X^\epsilon_t,Z^\epsilon_t)$ be the jump process of
intensity matrix~\eqref{eq:def_Q_eps} and starting from an initial
condition $Y_0=(X_0,Z_0)$ independent of $\epsilon$. 
We make the assumption~\eqref{eq:hyp}.
We denote by
$\P^\epsilon$ the 
distribution of the process $\left(Z^\epsilon_t\right)$ 
and by $\P$ the distribution of the jump process of initial condition
$Z_0$ and of intensity matrix 
\begin{equation}
\label{eq:def_Q_lambda}
\left(
\begin{array}{cc}
\ 0  & \lambda_0   \\
\ \lambda_1  & 0 
\end{array}
\right),
\end{equation}
where $\lambda_0$ and $\lambda_1$ are defined by~\eqref{eq:def_lambda}.
Then, we have $\mathcal P^\epsilon \Rightarrow  \mathcal P$ as $\epsilon$ goes
to $0$. 
\end{theorem}
Note that, in~\eqref{eq:def_Q_lambda}, we have used the convention
detailed in Remark~\ref{rem:convention}.

\medskip

The above result confirms the intuition according to which, when $\epsilon$
goes to zero, the internal dynamic within each macro-state is speeded
up, thus attaining a local equilibrium where configurations are
distributed according to the 
invariant measures $\pi_0$ and $\pi_1$ within the macro-states. In the
limit when $\epsilon$ goes to 0, the transition from one macro 
state $z$ to the other one, $1-z$, occurs with the frequency
$\lambda_z$, which is a weighted average (over the micro-states $x$,
with weights given by the 
invariant measure $\pi_z$) of the frequencies
$\overline{C}_{z,1-z}(x)$. In turn, these frequencies are the transition
frequencies from the micro-state $x$ of the macro-state $z$ to the other
macro-state.

As already emphasized in the introduction, we
point out that the above theorem states a convergence result on the {\em
  path} $(Z^\epsilon_t)_{t \geq 0}$, and not only of the random variable
$Z^\epsilon_t$ at any time $t$.

\subsection{Proofs}
\label{section_1_2}

To simplify the notation, we first consider the case when both
macro-states are similar: in that case, $Q_0=Q_1=Q$ and $C_{0,1}
=C_{1,0}=C$. The proof of Theorem~\ref{th1} is performed in
Section~\ref{sec:actual_proof}, and uses some intermediate results shown
in Section~\ref{sec:ingredients}. 
We briefly mention in Section~\ref{sec:proof_non_symmetric}
how to adapt the proof to handle the general case.

\medskip

The following computation will be very useful in what follows. 
Recall that the generator of the process
$Y^\epsilon_t$ is given by 
\begin{multline*}
L^\epsilon \phi\left(x,z\right)
=
\sum_{x'\in M} \epsilon^{-1}
Q\left(x,x'\right)\left(\phi\left(x',z\right)-\phi\left(x,z\right)\right)
\\
+ \sum_{x'\in M} C\left(x,x'\right)
\left(\phi\left(x',1-z\right)-\phi\left(x,z\right)\right).
\end{multline*}
We refer the reader to the textbook~\cite[Section 4.2]{kurtz}
for more details on semi-groups and generators associated to jump
processes. 

\medskip

Taking 
$\phi(x,z) = 1_{z=1}(x,z)$ in the above relation, we obtain
$$
L^\epsilon 1_{z=1}\left(x,z\right)
=- \sum_{x'\in M} C\left(x,x'\right) 1_{z=1}(x,z) 
+ 
\sum_{x'\in M}C\left(x,x'\right) 1_{z=0}(x,z),
$$ 
and thus, taking $(x,z) = Y^\epsilon_t = (X^\epsilon_t,Z^\epsilon_t)$,
we have
$$
L^\epsilon 1_{z=1}\left(Y^\epsilon_t\right) 
=  
\sum_{x'\in M} C\left(X^\epsilon_t,x'\right) \left(1-2 Z^\epsilon_t
\right)
=
\overline{C}\left(X^\epsilon_t\right) \left( 1-2 Z^\epsilon_t \right)
$$
where $\overline{C}(x) = \sum_{x'\in M} C\left(x,x'\right)$.
We now define the process $(M_t^\epsilon)_{t\geq 0}$ by 
\begin{eqnarray} 
\label{eqfonda}
M^\epsilon_{t}
&=&
1_{z=1}(Y^\epsilon_t) - 1_{z=1}(Y^\epsilon_0) - \int^t_0 
L^\epsilon 1_{z=1}\left(Y^\epsilon_s\right) \, ds
\nonumber
\\
&=&
Z^\epsilon_t - Z_0 - \int^t_0 \sum_{x'\in M} 
C\left(X^\epsilon_s,x'\right)\left(1-2 Z^\epsilon_s \right)ds.
\end{eqnarray} 
Using Proposition~\ref{lemma1}, we see that $M^\epsilon_t$ is a
martingale with respect to the filtration $\mathcal
F_t^{\epsilon}=\sigma\left(Y^\epsilon_s,s\leq t\right)$, and that its quadratic
variation is given by
\begin{eqnarray}
\langle M^{\epsilon}\rangle _t 
& = & 
\int^t_0 \left(L^{\epsilon}  1_{z=1} \left(Y^\epsilon_s\right) -2 \, 
1_{z=1} \left(Y^\epsilon_s\right) \, L^{\epsilon} 1_{z=1}
\left(Y_s^{\epsilon} \right )\right) ds 
\nonumber
\\
& =& 
\int^t_0 \overline{C}\left(X^{\epsilon}_s\right) \left(1-2
  Z^{\epsilon}_s\right) 
-2 Z_s^{\epsilon} \, \overline{C}\left(X_s^{\epsilon}\right) \, 
\left(1-2Z_s^{\epsilon}\right) ds 
\nonumber
\\
& =& 
\int^t_0 
\overline{C}\left(X_s^{\epsilon}\right) \,
\left(1-2Z^{\epsilon}_s\right)^2 ds
\nonumber
\\
& =& \int^t_0 \overline{C}\left(X_s^{\epsilon}\right)ds  
\nonumber
\\
&=& \int^t_0 g\left(X_s^{\epsilon}\right)ds +\lambda t, 
\label{eq:quad_var_M}
\end{eqnarray}
where $\lambda = \lambda_0 = \lambda_1$ (see~\eqref{eq:def_lambda}) and
\begin{equation}
\label{eq:def_g}
g\left(x\right)
=
\overline{C}\left(x\right)-\lambda
=
\sum_{x'\in M}C\left(x,x'\right)-\lambda.
\end{equation}
We have used in the above computation the fact that
$\left(1-2Z^{\epsilon}_s\right)^2 = 1$, a direct consequence of the fact
that $Z^\epsilon_s = 0$ or 1.

In what follows, we will use the fact that
\begin{equation} 
\label{ff}
Z^\epsilon_t
= 
Z_0 + \int^t_0 f\left(Y^\epsilon_s\right)ds 
+
\int^t_0\lambda\left(1-2 Z^\epsilon_s \right) \, ds+ M^\epsilon_t
\end{equation}
with
\begin{equation} 
\label{f}
f\left(x,z\right)
=\left(\sum_{x'\in M}C\left(x,x'\right)-\lambda\right) \ \left(1-2z\right),
\end{equation}
which is a straightforward reformulation of~\eqref{eqfonda}.

\subsubsection{Some intermediate results}
\label{sec:ingredients}

The following results are useful in the proof of Theorem~\ref{th1}.

\begin{lemma}
\label{lemme1}
Let $L$ be a $m \times m$ matrix and let $x\in \mathbb{R}^m$. Assume that for any $y \in \mathbb{R}^m$ such that
$y^T L=0$, we have $y^T x=0$. Then, there exists $z \in \mathbb{R}^m$ such
that $L z=x$.
\end{lemma}

\begin{proof}
We denote by $C_1$, \dots, $C_m$ the columns of the matrix $L$ and by
$V=\mbox{Span} \{C_1, \dots ,C_m\}$. The assumption on $x$ is that $x \in
\left(V^\perp\right)^\perp=V$. Therefore, $x$ can be written in the form
$x=\sum_{i=1}^m \alpha_i C_i$ for some coefficients $(\alpha_i)_{1 \leq
  i \leq m}$. Let $z=\sum_{i=1}^m \alpha_i e_i$, with
$\left(e_i\right)_{1\leq i \leq m}$ the canonical basis of
$\mathbb{R}^m$. We check that $z$ satisfies $Lz=x$. 
\end{proof}

\begin{lemma}
\label{lemme2}
Let $F= \{0, 1\}$, $Z_0$ be a random variable valued in $F$,
$\lambda_0,\lambda_1\geq 0$, and 
$\left(Z_t\right)_{t\geq 0}$ be a stochastic process on $F$. If the
process 
$$
M_t = Z_t-Z_0-\int_0^t \left( \lambda_0
  -\left(\lambda_0+\lambda_1\right) Z_s\right) ds
$$
is a martingale with respect to the natural filtration of
$\left(Z_t\right)_{t\geq 0}$, then $(Z_t)_{t\geq 0}$ is a Markov jump
process of initial condition $Z_0$ and of intensity matrix given by 
\begin{equation}
\label{Q_l0_l1}
R=\left(
\begin{array}{cc}
\ 0 & \lambda_0  \\
\ \lambda_1 & 0
\end{array}
\right).
\end{equation}
\end{lemma}

\begin{proof}
We use the uniqueness result of the martingale problem
associated to the Markov jump process with intensity matrix $R$
introduced by D.W. Stroock and S.R.S. Varadhan (see e.g. \cite[Theorem
21.11]{kallenberg}). We recall a simple version of that result in
Lemma~\ref{appendixG} below. In view of that result, we only need to
check that, for any bounded function $\phi : F \mapsto \RR$, the process
$$ 
M^\phi_t
=
\phi\left(Z_t\right)-\phi\left(Z_0\right) - \int_0^t L\phi\left(Z_s\right)ds 
$$
is a martingale, where $L$ is the generator of the jump process
associated to the intensity matrix~\eqref{Q_l0_l1}, which reads
$$
L \phi(z) = \sum_{z' \in F} R(z,z') \left( \phi(z')-\phi(z) \right).
$$
We note that
$$
L \phi(z=0) = \lambda_0 \left( \phi(1)-\phi(0) \right),
\quad
L \phi(z=1) = \lambda_1 \left( \phi(0)-\phi(1) \right).
$$
Since $F= \{0, 1\}$, any bounded function $\phi : F \mapsto \RR$ is of the
form 
$$
\forall z \in F, \quad
\phi(z) = a \delta_{0z} + b \delta_{1z} 
= 
a + \left(b-a\right)\delta_{1z},
$$
for some $a$ and $b$, where $\delta_{1z}$ is the Kronecker symbol. The
application 
$\phi \mapsto M^\phi_t$ is obviously linear, and it vanishes for constant
functions. Therefore, to show that $M^\phi_t$ is a martingale for any
bounded function $\phi : F \mapsto \RR$, it is
sufficient to show that $M^{\delta_{1z}}_t$ is a martingale. On $F$, we
see that $\delta_{1z} = \Id$. We thus have
\begin{eqnarray*}
M^{\delta_{1z}}_t 
&=& 
M^{\Id}_t
\\ 
&=& 
Z_t-Z_0- \int_0^t L \Id\left(Z_s\right)ds
\\
&=&  Z_t-Z_0-\int_0^t 
\left(\lambda_0-\left(\lambda_0+\lambda_1\right)Z_s\right)ds.
\end{eqnarray*}
Using the assumption of the Lemma, we have that $M^{\delta_{1z}}_t$ is a
martingale. This concludes the proof.
\end{proof}

\begin{lemma}
\label{lemme_continuite}
Let $g: \RR \rightarrow \RR$ be a Lipschitz function. Then, the
function $\Phi$ defined by 
\begin{eqnarray*}
\Phi  : D_{\mathbb R}\left[0,\infty\right) 
&\rightarrow & 
C_{\mathbb R}\left[0,\infty\right)
\subset D_{\mathbb R}\left[0,\infty\right)
\\
x & \mapsto & \left(\int_0^t g\left(x\left(s\right)\right)ds \right)_t
\end{eqnarray*}
is continuous.
\end{lemma}

\begin{proof}
Let $(x_n)_{n\in\NN}$ be a sequence in $D_{\mathbb
  R}\left[0,\infty\right)$ and $x$ in $D_{\mathbb
  R}\left[0,\infty\right)$ such that $(x_n)_{n\in\NN}$ converges to $x$
in $D_{\mathbb R}\left[0,\infty\right)$ for the Skorohod topology. We
show that $(\Phi(x_n))_{n\in\NN}$ converges to $\Phi(x)$ in the
Skorohod topology. 

We first observe that, for any $y\in D_{\mathbb
  R}\left[0,\infty\right)$, the function $\Phi(y)$ is continuous.
Since the limit function $\Phi(x)$ is hence continuous, the convergence
of $(\Phi(x_n))_{n\in\NN}$ to $\Phi(x)$ in the Skorohod topology is
equivalent to the convergence of $(\Phi(x_n))_{n\in\NN}$ to $\Phi(x)$
according to the norm $\| \cdot \|_{C^0([0,T])}$, on any compact time
interval $[0,T]$ (see e.g.~\cite[p.~124]{billy2}). 

\medskip

We now proceed and show that, for any $T>0$, 
$\| \Phi(x_n) - \Phi(x) \|_{C^0([0,T])}$ goes to zero as $n$ goes to
$\infty$. Using the characterization of the convergence
of $(x_n)_{n\in\NN}$ to $x$ given in Proposition~\ref{convergence_D}, we
know that there exists a sequence of strictly increasing, continuous
maps $\lambda_n$ defined on $[0,\infty)$ satisfying~\eqref{eq:convergence_D1}
and~\eqref{eq:convergence_D2} below. We then have, for any $t \in [0,T]$,
\begin{eqnarray}
\nonumber
&& \left| \Phi \left(x_n\right)\left(t\right) -
  \Phi\left(x\right)\left(t\right)\right| 
\\
\nonumber
&= & 
\left| \int^t_0 \left(g\left(x_n\left(s\right)\right) -
    g\left(x\left(s\right)\right)\right)ds \right|
\\
& \leq & 
\hspace{-3mm}
\int^t_0
\left| g\left(x_n\left(s\right)\right) -
  g\left(x\left(\lambda_n\left(s\right)\right)\right)\right|ds
+ 
\hspace{-2mm}
\int^t_0 \left|
  g\left(x\left(\lambda_n\left(s\right)\right)\right) - 
g\left(x\left(s\right)\right)\right|ds. 
\label{eq:titi}
\end{eqnarray}
The first term of the right-hand side of~\eqref{eq:titi} tends to $0$ as
$n$ goes to $\infty$ uniformly on $[0,T]$. Indeed,  
\begin{eqnarray*}
\sup_{t\in [0,T]} \int^t_0 \left| g\left(x_n\left(s\right)\right) -
g\left(x\left(\lambda_n\left(s\right)\right)\right)\right|ds
 & \leq & 
T \sup_{s\in [0,T]}
\left|g\left(x_n\left(s\right)\right) - 
g\left(x\left(\lambda_n\left(s\right)\right) \right)\right| 
\\
& \leq & T \, C_g \sup_{s\in [0,T]} 
\left| x_n\left(s\right) - x\left(\lambda_n\left(s\right)\right) \right|,
\end{eqnarray*}
where $C_g$ is the Lipschitz constant of
$g$. Using~\eqref{eq:convergence_D2}, we deduce that
\begin{equation}
\label{eq:titi1}
\lim_{n \to \infty} 
\sup_{t\in [0,T]} \int^t_0 \left| g\left(x_n\left(s\right)\right) -
g\left(x\left(\lambda_n\left(s\right)\right)\right)\right|ds
=
0.
\end{equation}
We now turn to the second term of the right-hand side of~\eqref{eq:titi}. Take
$\alpha>0$. Using~\cite[Lemma~1 p. 110]{billy}, we know that there
exists a subdivision 
$$
0=t_0 < t_1 < \cdots < t_r=T
$$ 
of $[0,T]$ such that, for any $i$, 
$$
\sup\{ |x\left(s\right) - x\left(t\right)|, \ t_i \leq s \leq t \leq
t_{i+1} \} \leq \alpha.
$$
This result is based on the fact that (i) a continuous function on a
compact set is also uniformly continuous on this set, and (ii) for any
$\beta>0$, a c\`ad-l\`ag function
on a compact set has a finite number of jumps larger than the threshold
$\beta$. 

Using this subdivision of $[0,T]$, we bound the second term of the
right-hand side of~\eqref{eq:titi} by 
\begin{eqnarray}
\int^t_0 \left| g\left(x\left(\lambda_n\left(s\right)\right)\right) - 
g\left(x\left(s\right)\right)\right|ds 
& \leq & 
\sum_{i=0}^{r-1} 
\int^{t_{i+1}}_{t_i} \left| 
g\left(x\left(\lambda_n\left(s\right)\right)\right) - 
g\left(x\left(s\right)\right)\right|ds 
\nonumber
\\
& \leq & 
\sum_{i=0}^{r-1} 
C_g \int^{t_{i+1}}_{t_i} \left| x\left(\lambda_n\left(s\right)\right) -
x\left(s\right)\right|ds. 
\label{eq:tutu1}
\end{eqnarray}
Let us introduce $\delta>0$ such that for any $0 \leq i \leq r-1$, we have 
$2 \delta < t_{i+1}-t_i$. As there is a finite number of points $t_i$,
such a $\delta>0$ exists. Using the property~\eqref{eq:convergence_D1}
of $\lambda_n$, we know that there exists $N$ such that, for any $n>N$, 
we have $\dps \sup_{s \in [0,T]} |\lambda_n\left(s\right)-s| \leq \delta$. We
therefore deduce that, for any $n>N$, 
\begin{eqnarray}
&& 
\sum_{i=0}^{r-1}
\int^{t_{i+1}}_{t_i} | x\left(\lambda_n\left(s\right)\right) - 
x\left(s\right)|ds 
\nonumber
\\
&\leq &
\sum_{i=0}^{r-1}
\int^{t_{i+1}-\delta}_{t_i+\delta} |
x\left(\lambda_n\left(s\right)\right) -  x\left(s\right)|ds
+
4 r \delta \sup_{t\in [0,T+\delta]} |x\left(t\right)| 
\nonumber
\\
& \leq& 
\sum_{i=0}^{r-1}
\left( t_{i+1} - t_i - 2 \delta \right) \alpha +
4r\delta \sup_{t\in [0,T+\delta]} |x\left(t\right)| 
\nonumber
\\
& \leq& T \alpha + 4r\delta \sup_{t\in [0,T+\delta]} |x\left(t\right)|. 
\label{eq:tutu2}
\end{eqnarray}
Inserting~\eqref{eq:tutu2} in~\eqref{eq:tutu1}, we deduce that 
the second term of the right-hand side of~\eqref{eq:titi} is bounded by
$$
\int^t_0 \left| g\left(x\left(\lambda_n\left(s\right)\right)\right) - 
g\left(x\left(s\right)\right)\right|ds 
\leq 
C_g T \alpha + 4 C_g r\delta \sup_{t\in [0,T+\delta]} |x\left(t\right)|.
$$
As $\alpha$ and $\delta$ are arbitrary small, and $r$ only depends on
$\alpha$, we conclude that the second term of the right-hand side
of~\eqref{eq:titi} converges to $0$ uniformly in $t$ on
$[0,T]$. 

Collecting this result with the limit~\eqref{eq:titi1} on the first
term and~\eqref{eq:titi}, we deduce that
$$
\lim_{n \to \infty} \sup_{t\in [0,T]} 
\left| \Phi \left(x_n\right)\left(t\right) -
  \Phi\left(x\right)\left(t\right)\right| 
=
0.
$$
This concludes the proof of Lemma~\ref{lemme_continuite}.
\end{proof}

\begin{remark}  
If the function $g$ is not continuous, then $\Phi$ is not
continuous. Consider indeed a sequence $(x_n)_{n\in\NN}$ of real numbers
that converges from above to $x$, a discontinuity point of $g$. Denoting
$\Phi\left(x_n\right)$ the image by $\Phi$ of the constant function
equal to $x_n$, we see that, for any $t$, 
$$
\Phi\left(x_n\right)(t) -\Phi\left(x\right)(t)
\ \longrightarrow \
t\left(g\left(x+\right)-g\left(x\right)\right) \neq 0.
$$
\end{remark}

We conclude these intermediate results with the following proposition,
that will be useful to study the limit when $\epsilon \to 0$ of the
second term in the right-hand side of~\eqref{ff}. 
\begin{proposition}
\label{prop1}
Let $f$ be given by~\eqref{f}. 
Under the hypothesis of Theorem~\ref{th1}, we have, for any $t \geq 0$,
\begin{equation}
\label{eq:prop1}
\mathbb{E} \left[ \left( \int^t_0 f\left(Y^\epsilon_s\right) ds
  \right)^2 \right] 
\longrightarrow  0 \ \ \mbox{as $\epsilon \rightarrow 0$}. 
\end{equation}
\end{proposition}

\begin{proof}
Since $E$ is a finite set, we identify functions $\phi:E\rightarrow \RR$
with the vectors $\bra{\bra{\phi(x,0)}_{x\in M},\bra{\phi(x,1)}_{x\in
    M}}\in \RR^{2m}$ throughout the proof. We likewise identify
operators with matrices.

Let $\overline{L}^0$ be the generator corresponding to the intensity matrix
$\overline{Q}^0$:
$$
\overline{L}^0 u\left(x,z\right) = \sum_{x' \in M} Q\left(x,x'\right)
\left( u\left(x',z\right) - u\left(x,z\right)\right).
$$
First, we claim that
\begin{equation}
\label{eq:claim}
\text{there exists a function $u : E \mapsto \RR$ such that
$\overline{L}^0 u=f$}.
\end{equation}
Indeed, as $Q$ is irreducible, the only
vectors $\mu \in \RR^{2m}$ such that $\mu^T \overline{L}^0=0$ are the vectors of
the form $\mu_{\alpha, \beta}=\left(\alpha \pi, \beta \pi\right)$ for
any $\alpha, \beta \in \mathbb R$ (this is a simple consequence of the
Perron-Frobenius theorem). 
Using~\eqref{f} and~\eqref{eq:def_lambda}, we compute
\begin{eqnarray}
\mu_{\alpha, \beta}^Tf  
&= & 
\sum_{x\in M} \alpha \pi\left(x\right) f\left(x,0\right) 
+ \sum_{x\in M} \beta \pi\left(x\right) f\left(x,1\right) 
\nonumber
\\        
&= &
\left(\alpha-\beta\right) \left( \sum_{x\in M} \pi\left(x\right) 
\sum_{x'\in M} C\left(x,x'\right)-\lambda\right) 
\nonumber
\\        
& = & 0.
\label{moyenne_nulle}
\end{eqnarray}
We thus see that, for any $\mu \in \RR^{2m}$ such that $\mu^T
\overline{L}^0=0$, we have $\mu^T f=0$. We are now in position to use
Lemma~\ref{lemme1}, from which we deduce the claim~\eqref{eq:claim}. 

\medskip

Second, using~\eqref{eq:claim}, we write that
\begin{eqnarray} 
\int^t_0 f\left(Y^\epsilon_s\right)ds 
& = & 
\int^t_0 \overline{L}^0 u\left(Y^\epsilon_s\right) ds 
\nonumber 
\\
& = & 
\epsilon \int^t_0 L^{\epsilon} u\left(Y^\epsilon_s\right) - \epsilon
\int^t_0 L^C u\left(Y^\epsilon_s\right) ds
\label{lint}
\end{eqnarray}
where we have used the decomposition
$$
\epsilon L^{\epsilon} u= \overline{L}^0u+\epsilon L^Cu
$$
with
$$
L^Cu\left(x,z\right) = \sum_{x' \in M} C\left(x,x'\right) 
\left( u\left(x',1-z\right) - u\left(x,z\right)\right).
$$
We successively bound the two terms of the right-hand side of~\eqref{lint}.
Introduce 
$\dps N^u_t = u\left(Y^\epsilon_t\right) - u\left(Y^\epsilon_0\right) -
\int^t_0 L^{\epsilon} u\left(Y^\epsilon_s\right) ds$. In view of
Proposition~\ref{lemma1}, we know that $N^u_t$ is a martingale of
quadratic variation given by  
$$
\langle N^u\rangle_t
=
\int^t_0 \left(L^{\epsilon}u^2\left(Y^\epsilon_s\right) - 2
  u\left(Y^\epsilon_s\right) \,
  L^{\epsilon}u\left(Y^\epsilon_s\right)\right) \, ds.
$$
For any $v:E \rightarrow \mathbb{R}$, we have
$$
\|L^{\epsilon}v\|_{\infty}
\leq
2m \|v\|_\infty \left( \epsilon^{-1} \|Q\|_\infty + \|C\|_\infty \right).
$$
Therefore, 
\begin{eqnarray*}
\mathbb{E} \left[ \left(N^u_t\right)^2 \right]  
& = & 
\mathbb{E}\left(\langle N^u\rangle_t\right) 
\\
& \leq & 
2 m t \left[ \|u^2\|_\infty \left( \epsilon^{-1} \|Q\|_\infty +
    \|C\|_\infty \right) + 2 \|u\|^2_\infty \left( \epsilon^{-1}
    \|Q\|_\infty + \|C\|_\infty \right) \right]
\\
& \leq & A + \epsilon^{-1} B,
\end{eqnarray*}
where $A$ and $B$ are positive constants independent of $\epsilon$. It
follows that the first term of the right hand side of~\eqref{lint}
satisfies
\begin{eqnarray}
\mathbb{E}\left[ \left( \epsilon \int^t_0 L^\epsilon
    u\left(Y^\epsilon_s\right)\right)^2 \right] 
& = & 
\mathbb{E}\left[ \left( \epsilon \left( N^u_t -
      u\left(Y^\epsilon_t\right) + u\left(Y^\epsilon_0\right) \right)
  \right)^2 \right]
\nonumber  
\\
& \leq & 
2 \epsilon^2 \left( \mathbb{E}\left[ \left(N^u_t\right)^2 \right] + 4
  \|u^2\|_\infty \right) 
\nonumber  
\\
& \leq & 2 \epsilon^2 \left( A' +\epsilon^{-1} B \right).
\label{eq:titi2} 
\end{eqnarray}
For the second term of the right hand side of~\eqref{lint}, we directly
obtain
\begin{equation}
\mathbb{E}\left[\left(\epsilon \int^t_0 L^C
    u\left(Y^\epsilon_s\right)\right)^2 \right] 
\leq 
\epsilon^2 t^2 \left( 4m \|C\|^2_\infty \|u\|^2_\infty \right).
\label{eq:titi3} 
\end{equation}
Collecting~\eqref{lint}, \eqref{eq:titi2} and~\eqref{eq:titi3}, we
obtain the desired result~\eqref{eq:prop1}. 
This concludes the proof of Proposition~\ref{prop1}.
\end{proof}


\subsubsection{Proof of Theorem~\ref{th1} (symmetric case)}
\label{sec:actual_proof}

All the convergences in this proof are taken when $\epsilon$ goes to
$0$. We will omit to recall it. The proof consists of four steps.  

\medskip
\noindent
\textbf{Step 1: the family of probability measures $\left(\mathcal
    P^\epsilon\right)_{\epsilon>0}$ is relatively compact} 

We use the tightness criterion of Theorem~\ref{critere}, and check
that its conditions~\eqref{3.21i} and~\eqref{4.4} are satisfied.

As the variables $Z^\epsilon_t$ take only two values, 0 and 1,
the condition~\eqref{3.21i} is trivially satisfied with the choices
$K=1$ and $n_0=1$. 

Let us now show that the condition~\eqref{4.4} is
satisfied. Let $N \in \mathbb N$, $\alpha >0$, $\theta > 0 $ and
$\epsilon > 0$. Let $S$ and $T$ be two $\mathcal F^\epsilon$-stopping
times such that $S\leq T \leq S + \theta \leq N$. Recall that a random
variable $T:\bra{\O,\bra{\F_t}_{t\geq 0}}\rightarrow
\RR^+\cup\set{\infty}$ is a stopping time if, for any $t\geq 0$, the set
$\set{T\leq t}$ is $\F_t$-measurable. Using~\eqref{eqfonda}, we have
\begin{equation}
\label{eq:titi4}
\left| Z^\epsilon_T - Z^\epsilon_S \right|  
\leq  
\left| \int^T_S \sum_{y\in M} C\left(X^\epsilon_s,y\right) 
\left(1-2 Z^\epsilon_s \right) ds \right| + 
\left| M^\epsilon_T - M^\epsilon_S \right|. 
\end{equation}
The first term of the right-hand side of~\eqref{eq:titi4} is bounded as
follows: 
\begin{equation} 
\label{terme1}
\left| \int^T_S \sum_{y\in M} C\left(X^\epsilon_s,y\right) \left(1-2
    Z^\epsilon_s \right) ds \right| 
\leq 
\left| T -S \right| m \| C \|_\infty
\leq 
\theta m \| C \|_\infty.
\end{equation}
To bound the second term of the right-hand side of~\eqref{eq:titi4}, we
use the Tchebytchev inequality: 
\begin{equation} 
\label{eq:titi5}
\mathbb P\left( \left| M^\epsilon_T - M^\epsilon_S \right| 
\geq \alpha \right)
\leq 
\frac{\mathbb E \left| M^\epsilon_T - M^\epsilon_S \right|^2}{\alpha^2}.
\end{equation}
We denote by $\widetilde{M}_t^\epsilon = M^\epsilon_{t+S} - M^\epsilon_S$
and $\widetilde{\mathcal F}^\epsilon_t = \mathcal F^\epsilon_{t+S}$. As
$S$ is a bounded stopping time, we infer from the optional stopping 
theorem (see e.g.~\cite[Theorem 3.2]{revuz}) that
$\widetilde{M}^\epsilon$ is a $\widetilde{\mathcal
  F}^\epsilon$-martingale, of quadratic variation
$$
\langle \widetilde{M}^\epsilon \rangle_t
=
\langle M^\epsilon \rangle_{S+t} - \langle M^\epsilon \rangle_S.
$$
In particular, we have
$$
\langle \widetilde{M}^\epsilon \rangle_{T-S} =
\langle M^\epsilon \rangle_T - \langle M^\epsilon \rangle_S.
$$
It follows that
\begin{eqnarray}
\mathbb E \left[ | M^\epsilon_T - M^\epsilon_S |^2 \right] 
&=& 
\mathbb E \left[ | \widetilde{M}^\epsilon_{T-S} |^2 \right] 
\nonumber 
\\
& =& 
\mathbb E \left[ \langle \widetilde{M}^\epsilon\rangle _{T-S} \right] 
\nonumber 
\\
& =& 
\mathbb E \left[ \langle M^\epsilon\rangle _{T}-\langle M^\epsilon\rangle _S\right]
\nonumber  
\\
& =& 
\mathbb E \left[ \int^T_S g\left(X_s^\epsilon\right)ds +\lambda
  \left(T-S\right)\right]
\nonumber  
\\
& \leq & 
\theta \left( \|g\|_{\infty}+\lambda \right),
\label{maj_vq}
\end{eqnarray}
where we have used~\eqref{eq:quad_var_M} and where $g$ is defined
by~\eqref{eq:def_g}. We then infer from~\eqref{eq:titi5} that
\begin{equation} 
\label{eq:titi6}
\mathbb P\left( \left| M^\epsilon_T - M^\epsilon_S \right| 
\geq \alpha \right)
\leq 
\frac{\theta \left( \|g\|_\infty + \lambda \right)}{\alpha^2}.
\end{equation} 
We deduce from~\eqref{eq:titi4}, \eqref{terme1} and~\eqref{eq:titi6}
that the condition~\eqref{4.4} of Theorem~\ref{critere} below is
satisfied. 

\medskip

Assumptions~\eqref{3.21i} and~\eqref{4.4} being satisfied, we can apply
Theorem~\ref{critere}, which implies that the family of probability
measures $(\P^\epsilon)_\epsilon$ is tight. In view of Prohorov's
theorem (see e.g.~\cite[Theorem 2.2]{kurtz}), this implies that 
the family $\left(\mathcal P^\epsilon\right)_{\epsilon>0}$ is
relatively compact. 

There thus exists a sub-family of
$\left(\mathcal P^\epsilon\right)_\epsilon$, that we denote
$\left(\mathcal P^{\epsilon'}\right)_{\epsilon'}$, which is
convergent. Otherwise stated, there exists a process $Z$ such that
$Z^{\epsilon'}\Rightarrow Z$. 

\medskip
\noindent
\textbf{Step 2: there exists a martingale $M_t$ and a sub-family
  $M_t^{\epsilon'}$ such that $M_t^{\epsilon'} \Rightarrow M_t$}

In view of~\cite[Theorem VI.4.13]{jacod}, a sufficient criterion for
$\left(M^\epsilon\right)$ to be relatively compact is that
$\left(\langle M^\epsilon\rangle \right)$ is C-tight. Let us check this
criterion. We have shown above (see~\eqref{eq:quad_var_M}) that
$$
\langle M^{\epsilon}\rangle _t 
=
\int^t_0 g\left(X_s^{\epsilon}\right)ds +\lambda t, 
$$
where $g$ is defined by~\eqref{eq:def_g}. Therefore, the family of paths
$\left(\langle M^\epsilon\rangle \right)_{\epsilon>0}$ is uniformly
Lipschitz, and hence C-tight (see~\cite[Definition VI.3.25 and
Proposition VI.3.26]{jacod}). We can thus consider a sub-family of 
$\left(M_t^\epsilon\right)_{t\geq 0}$, that we denote
$\left(M_t^{\epsilon'}\right)_{t\geq 0}$, which weakly converges to a
process $M$. Using~\cite[Proposition~IX.1.1]{jacod}, we know that the
process $(M_t)_{t\geq 0}$ is a martingale with respect to its natural
filtration.

\medskip
\noindent
\textbf{Step 3: equation satisfied by $Z$}

We have shown at the end of Step 1 that there exists a process $Z$ and a
sub-family $Z^{\epsilon'}$ such that $Z^{\epsilon'} \Rightarrow Z$. We
now identify a stochastic differential equation satisfied by
$(Z_t)_{t\geq 0}$. 

Recall first that $(Z_t^\epsilon)_{t\geq 0}$ satisfies~\eqref{ff}, namely
\begin{equation} 
\label{EDS1_pre}
Z_t^\epsilon = Z_0 + \int_0^t f\left(Y^\epsilon_s\right)ds 
+ \int^t_0 \lambda \left(1-2Z_s^\epsilon\right)ds +M_t^\epsilon.
\end{equation}
Passing to the limit $\epsilon' \to 0$, let us show that $(Z_t)_{t\geq 0}$
satisfies
\begin{equation} 
\label{EDS1}
Z_t=Z_0 + \int^t_0 \lambda\left(1-2Z_s\right)ds + M_t.
\end{equation}
We first consider $\dps B_t^\epsilon = \int_0^t
f\left(Y^\epsilon_s\right)ds$. With the same techniques as above, we can
show that $\left(B_t^\epsilon\right)$ is a relatively compact family. There
thus exists $\left(B_t\right)$ and a sub-family
$\left(B_t^{\epsilon'}\right)$ such that $B^{\epsilon'} \Rightarrow B$.
We infer from Proposition~\ref{prop1} that, for all $t\geq 0$,
$B_t^\epsilon$ converges to $0$ in $L^2(\O)$, hence $\mathbb E \left[
  B_t^2 \right]=0$ for all $t\geq 0$. It follows that the family
$\left(B_t^{\epsilon'}\right)$ converges to $0$ in distribution. 

We next turn to $\dps J_t^\epsilon = \int_0^t \lambda \left(1-2
  Z^\epsilon_s\right)ds$. Introduce $\dps J_t = \int_0^t \lambda
\left(1-2 Z_s\right)ds$. The function $g : z \mapsto \lambda \left(1-2
  z\right)$ is Lipschitz on $\mathbb R$, thus, using
Lemma~\ref{lemme_continuite}, we know that the function
\begin{eqnarray*} 
\Phi :D_{\mathbb R}\left[0,\infty\right) 
&\longrightarrow& 
D_{\mathbb R}\left[0,\infty\right) 
\\
z & \mapsto & \left(\int_0^t \lambda \left( 1 - 2 z\left(s\right)
  \right) ds \right)_t
\end{eqnarray*}
is continuous. The convergence $Z^{\epsilon'} \Rightarrow Z$ therefore
implies that 
$$
J^{\epsilon'} = \Phi(Z^{\epsilon'}) \Rightarrow \Phi(Z) = J.
$$
We have thus obtained that all the terms in~\eqref{EDS1_pre} weakly
converge. It remains to show that we can add up the weak limits. To do
so, we show with the same techniques as before that the family
$\left(B^\epsilon,J^\epsilon,M^\epsilon\right)$ is relatively compact, and that
the limit of any sub-family has as marginal distributions those of $B$,
$J$ and $M$. We conclude that 
$B^{\epsilon'}+J^{\epsilon'}+M^{\epsilon'} \Rightarrow B+J+M$. Passing
to the limit $\epsilon' \to 0$ in~\eqref{EDS1_pre}, we then indeed
obtain~\eqref{EDS1}. 

\medskip
\noindent
\textbf{Step 4: conclusion}

We infer from~\eqref{EDS1} (where, we recall, $M_t$ is a martingale) and
Lemma~\ref{lemme2} (with $\lambda_0=\lambda_1=\lambda$) that $\left( Z_t
  \right)_{t \geq 0}$ is a Markov jump process of initial condition
  $Z_0$ and of intensity matrix given
$$
\bra{
\begin{array}{cc}
  0 &\lambda\\
  \lambda & 0
\end{array}
}.
$$
The process $Z$ is thus uniquely defined. 

It follows that all convergent sub-families $Z^{\epsilon'}$ have the
same limit $Z$. The whole sequence $Z^\epsilon$ therefore converges to
this common limit $Z$. This concludes the proof of
Theorem~\ref{th1} in the symmetric case.

\subsubsection{Non-symmetric case}
\label{sec:proof_non_symmetric}

In this Section, we briefly sketch the proof in the non-symmetric case,
that is when 
$Q_0 \neq Q_1$ or $C_{0,1} \neq C_{1,0}$ in~\eqref{eq:def_Q_eps}. The
structure of the proof is similar to that in the symmetric case. 

First, the generator associated to the process $\bra{Y^\epsilon_t}$ reads
\begin{multline*}
L^\epsilon \phi\left(x,z\right)
=
\sum_{x'\in M} \epsilon^{-1} Q_z\left(x,x'\right)
\left(\phi\left(x',z\right) - \phi\left(x,z\right)\right) 
\\
+ \sum_{x'\in M} C_{z,1-z}\left(x,x'\right) 
\left(\phi\left(x',1-z\right) - \phi\left(x,z\right)\right).
\end{multline*}
Choosing the function $\phi\left(x,z\right)=z$, we see that
$$
L^\epsilon \phi\left(x,z\right) =  
\sum_{x'\in M}C_{z,1-z}\left(x,x'\right)\left(1-2z\right)=f(x,z)+h(z),
$$
where we have introduced (recall~\eqref{eq:def_lambda})
$$
f\left(x,z\right)
=\bra{\sum_{x'\in M}C_{z,1-z}\left(x,x'\right)-\lambda_z}\left(1-2z\right)
$$
and
$$
h(z) 
= 
(1-2z) \lambda_z
=
(1-2z) \sum_{x,x' \in M} C_{z,1-z}(x,x') \pi_z(x).
$$
Using again Proposition~\ref{lemma1}, we see that the process
\begin{eqnarray}
M^\epsilon_t
&=&
\phi(Y^\epsilon_t) - \phi(Y^\epsilon_0) - \int_0^t L^\epsilon
\phi(Y^\epsilon_s) \, ds
\label{eqje}
\end{eqnarray}  
is a martingale. Using the above notation, the equation~\eqref{eqje} can be recast as
\begin{equation}
\label{eqje_bis}
Z^\epsilon_t
= 
Z_0 + \int^t_0 f\left(Y^\epsilon_s\right)ds 
+ \int^t_0 h\left(Z^\epsilon_s\right)ds + M^\epsilon_t.
\end{equation}
To pass to the limit $\epsilon \to 0$ in the above equation, we follow
the same lines as in the proof detailed in Sections~\ref{sec:ingredients}
and~\ref{sec:actual_proof}. 

\medskip

Consider the second term of the right-hand side of~\eqref{eqje_bis}.
As in the proof of Proposition~\ref{prop1}, we can show that 
$\mu^T f=0$ for any $\mu \in \RR^{2m}$ such that $\mu^T \overline{L}^0=0$, which
are vectors of the form 
$\left(\alpha \pi_0,\beta \pi_1 \right)$ for any $\alpha$ and $\beta$ in
$\RR$. This implies that
$\dps \int^t_0 f\left(Y^\epsilon_s\right)ds$ converges to $0$ in
$L^2\left(\O\right)$ for any $t\geq 0$. 

\medskip

We turn now to the third term of the right-hand side
of~\eqref{eqje_bis}. Let $\widetilde{h}$ be the affine function defined
on $\RR$ by $\widetilde{h}(0) = h(0)$ and $\widetilde{h}(1) = h(1)$. The
function $\widetilde{h}$ is obviously Lipschitz on $\RR$, hence, using
Lemma~\ref{lemme_continuite}, we know that the function 
\begin{eqnarray*}
\Phi  : D_{\mathbb R}\left[0,\infty\right) 
&\rightarrow& 
D_{\mathbb R}\left[0,\infty\right)
\\
z & \mapsto & \left(\int_0^t \widetilde{h}(z(s)) \, ds \right)_t
\end{eqnarray*}
is continuous. Since $\dps \int^t_0 h\left(Z^\epsilon_s\right)ds
= \int^t_0 \widetilde{h}\left(Z^\epsilon_s\right)ds$, this allows to
pass to the limit in that term.

\medskip

As in Section~\ref{sec:actual_proof} (Step 3 of the proof), we can thus
pass to the limit $\epsilon \to 0$ in~\eqref{eqje_bis}, and show that
$Z^\epsilon$ 
converges in distribution to a process $Z$, that satisfies
\begin{eqnarray*}
Z_t
&=& 
Z_0 + \int^t_0 h(Z_s) ds + M_t
\\
&=&
Z_0 + \int^t_0 \left[ \lambda_0 -Z_s \left( \lambda_0 + \lambda_1 \right)
\right] ds + M_t,
\end{eqnarray*}
where $M$ is a martingale. We then infer from Lemma~\ref{lemme2} that
$\bra{Z_t}_{t\geq 0}$ is a jump process on $\{0, 1\}$, of initial
condition $Z_0$ and of intensity matrix
$$
\left(
\begin{array}{cc}
\ 0& \lambda_0
\\
\ \lambda_1& 0
\end{array}
\right),
$$
as claimed in Theorem~\ref{th1}.

\subsection{Numerical illustration}
\label{1}

We have implemented the model presented in Section~\ref{section_1_1}.
As shown on Fig.~\ref{fig:puits}, the energy wells can be gathered in two
macro-states (each of them containing $m$ 
micro-states) separated by a high potential energy barrier. The
transitions are only possible from one well to its two nearest
neighbours. In addition, we apply periodic
boundary conditions. The matrices $Q_0$, $Q_1$, $C_{0,1}$ and $C_{1,0}$
of the intensity matrix~\eqref{eq:def_Q_eps} read
$$
Q_0=Q_1=Q 
\quad \text{and} \quad
C_{0,1}=C_{1,0}=C
$$
with
$$
Q = \left(
\begin{array}{ccccc}
0 & q & & &  \\
q & 0 & q & & \\
\ddots & \ddots & \ddots & \\
 & &q &0 & q \\
 & & & q & 0 \\ 
\end{array}
\right),
\quad
C = \left(
\begin{array}{ccccc}
0 & \cdots&  &0  &  c  \\
0 & \cdots &  &  & 0\\
\vdots & & & & \vdots \\
c & 0 & \cdots&  & 0 \\ 
\end{array}
\right).
$$
We work with $q=c=1$. 

We are interested in the distribution of the first exit time
$S^\epsilon_0$ from a macro-state. From Theorem~\ref{th1}, we know that,
in the limit $\epsilon$ going to $0$, $S^\epsilon_0$ follows an exponential
distribution of parameter $\lambda=2c/m$ (independently of what the
initial condition of the system is). In order to quantify the
convergence of the 
distribution of $S^\epsilon_0$ to the predicted distribution, we
consider the $L^1$ norm of the difference of the densities: 
\begin{equation}
\label{eq:err_L1}
\text{err}_{L^1}
=
\int_0^\infty |f-f^\epsilon|
\approx 
\frac{1}{n} \sum_{i=1}^n 
|f\left(i\Delta x\right)-f^\epsilon_i|,
\end{equation}
where $f\left(x\right)=\lambda e^{-\lambda x}$ is the limit
distribution and $f^\epsilon$ is the distribution of
$S_0^\epsilon$. This latter distribution is calculated on the bounded
interval $[0,s]$ with $s=n \Delta x$ on a grid of size $\Delta x$:
$\dps f_i^\epsilon \approx \frac{1}{\Delta x} 
\sum_{x\in[i\Delta x, (i+1)\Delta x]}f^\epsilon(x)$ for any $i \in [1,n]$. 
In the sequel, we work with $\Delta x = 0.05$ and $s=n \Delta x =5$.

\begin{remark}
Other criteria can also be considered to characterize the convergence of the
probability distribution $f^\epsilon$ towards $f$. One example is the
discrepancy, which is the difference (in $L^\infty$ norm) of the
cumulative distribution functions:
\begin{equation}
\label{eq:discrepancy}
D=\sup_{A\geq 0} \left| \int_0^A f - \int_0^A f^\epsilon \right|.
\end{equation}
We have used this criterion e.g. on Fig.~\ref{faux} 
below. 
\end{remark}
 
We first consider how results depend on $\epsilon$. We work with a fixed
initial condition, namely $Y_0=\left(0,0\right)$. At the initial time,
the particle is in the first macro-state, and in the micro-state which is
the closest to the energy barrier between the two macro-states (see
Fig.~\ref{fig:puits}).

On
Figs.~\ref{moyenne_carryon} and~\ref{variance_carryon}, we show the
convergence of the empirical expectation and variance of $S^\epsilon_0$
to the asymptotic value (we have considered $10^4$ independent and
identically distributed realizations of the process to compute
95 \% confidence intervals). We indeed
observe convergence of both quantities to their asymptotic limits when
$\epsilon \to 0$. 

\begin{figure}[h]
\begin{center}
\psfrag{moyenne de S0}{}
\psfrag{'moyenne_ene.txt'}{\tiny \hspace{0.7cm} mean}
\psfrag{1000}{\tiny $10^3$}
\psfrag{100}{}
\psfrag{0.1}{}
\psfrag{0.001}{}
\psfrag{epsilon}{\tiny \hspace{-3.5cm} epsilon ($m=3$) \hspace{0.7cm} epsilon ($m=5$) \hspace{1cm} epsilon ($m=7$)}
\includegraphics[scale=0.75]{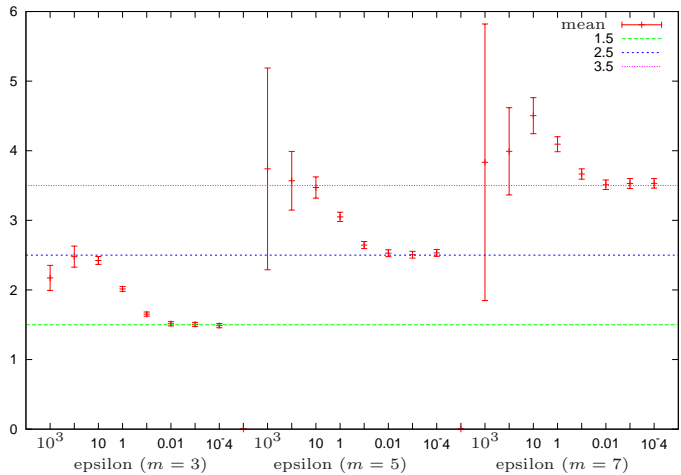}
\end{center}
\caption{Empirical expectation of $S^\epsilon_0$ as a function of
  $\epsilon$, for $m=3$ (left), $m=5$ (middle) and $m=7$ (right). The
  asymptotic values (when $\epsilon \to 0$) are also represented (solid
  lines).
\label{moyenne_carryon}}
\end{figure}

\begin{figure}[h]
\begin{center}
\psfrag{variance de S0}{}
\psfrag{'variance_ene.txt'}{\tiny \hspace{0.3cm} variance}
\psfrag{1000}{\tiny $10^3$}
\psfrag{100}{}
\psfrag{0.1}{}
\psfrag{0.001}{}
\psfrag{epsilon}{\tiny \hspace{-3.5cm} epsilon ($m=3$) \hspace{0.7cm} epsilon ($m=5$) \hspace{1cm} epsilon ($m=7$)}
\includegraphics[scale=0.75]{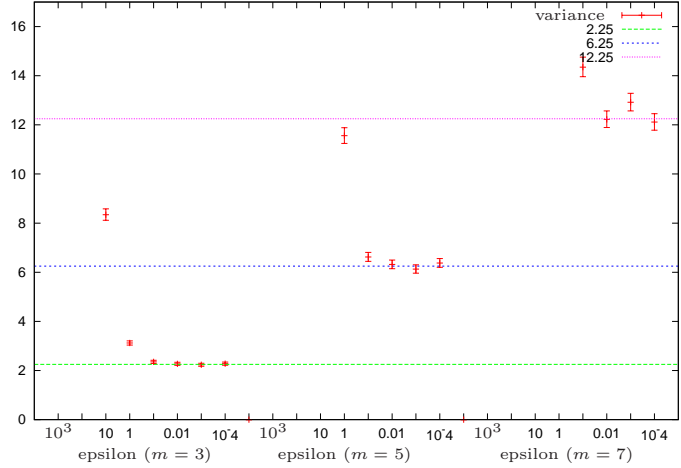}
\end{center}
\caption{Empirical variance of $S^\epsilon_0$ as a function of
  $\epsilon$, for $m=3$ (left), $m=5$ (middle) and $m=7$ (right). The
  asymptotic values (when $\epsilon \to 0$) are also represented (solid
  lines).
\label{variance_carryon}}
\end{figure}

On Fig.~\ref{hist_3}, we show the histogram of $S^\epsilon_0$ in the
case $m=20$ 
for two values of $\epsilon$. We again oberve a good
qualitative agreement with the limit distribution for small enough
$\epsilon$. This can be quantified by looking precisely at the
convergence of the distribution of $S^\epsilon_0$ to the asymptotic
distribution when $\epsilon$ goes to 0, for different values of $m$ (see
Fig.~\ref{faux}). The left part of that figure seems to show that
the convergence slows down when the number $m$ of micro-states within a
macro-state increases. 

\begin{figure}[h!]
\centerline{
\includegraphics[scale=0.4]{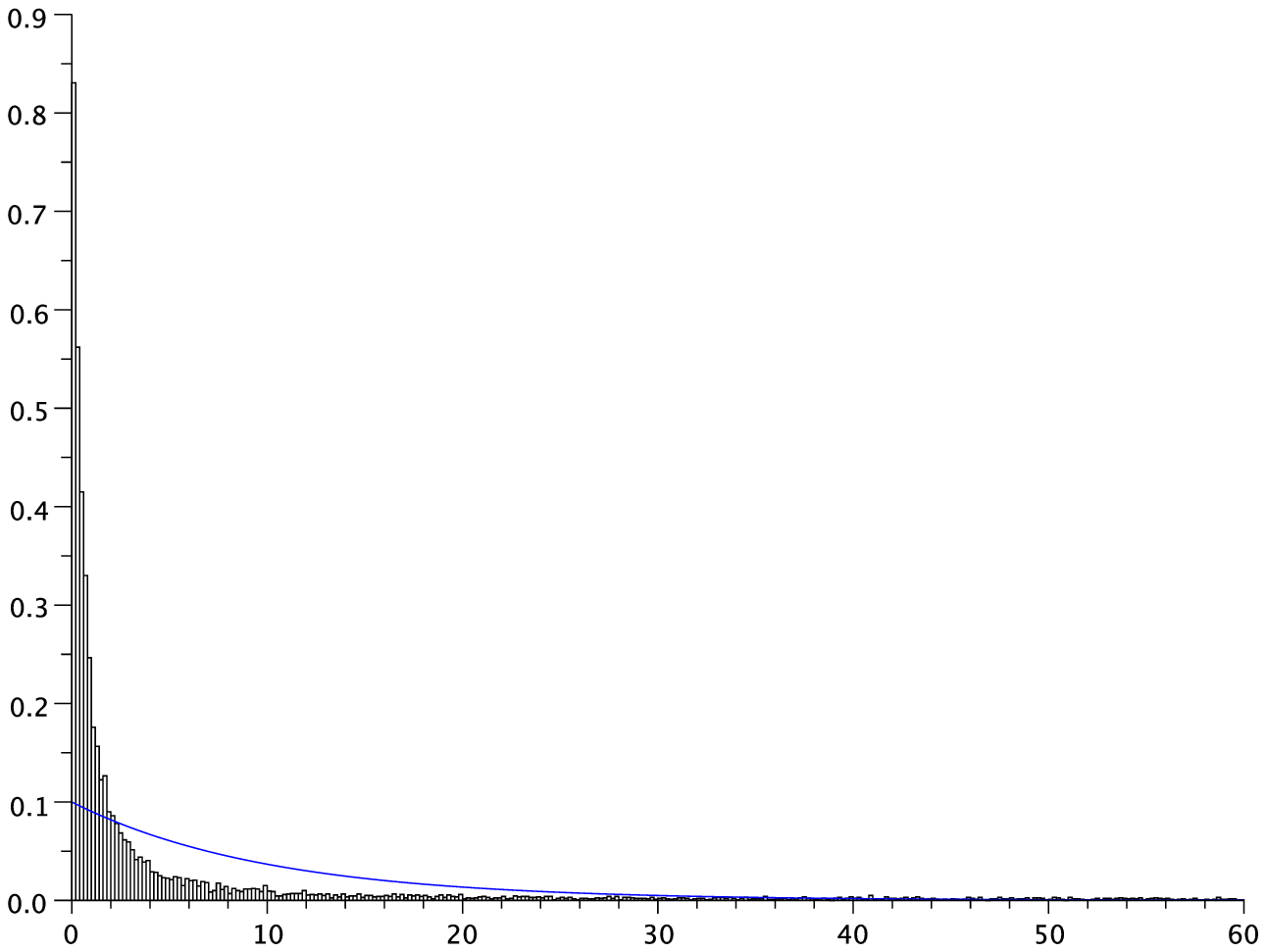}
\includegraphics[scale=0.4]{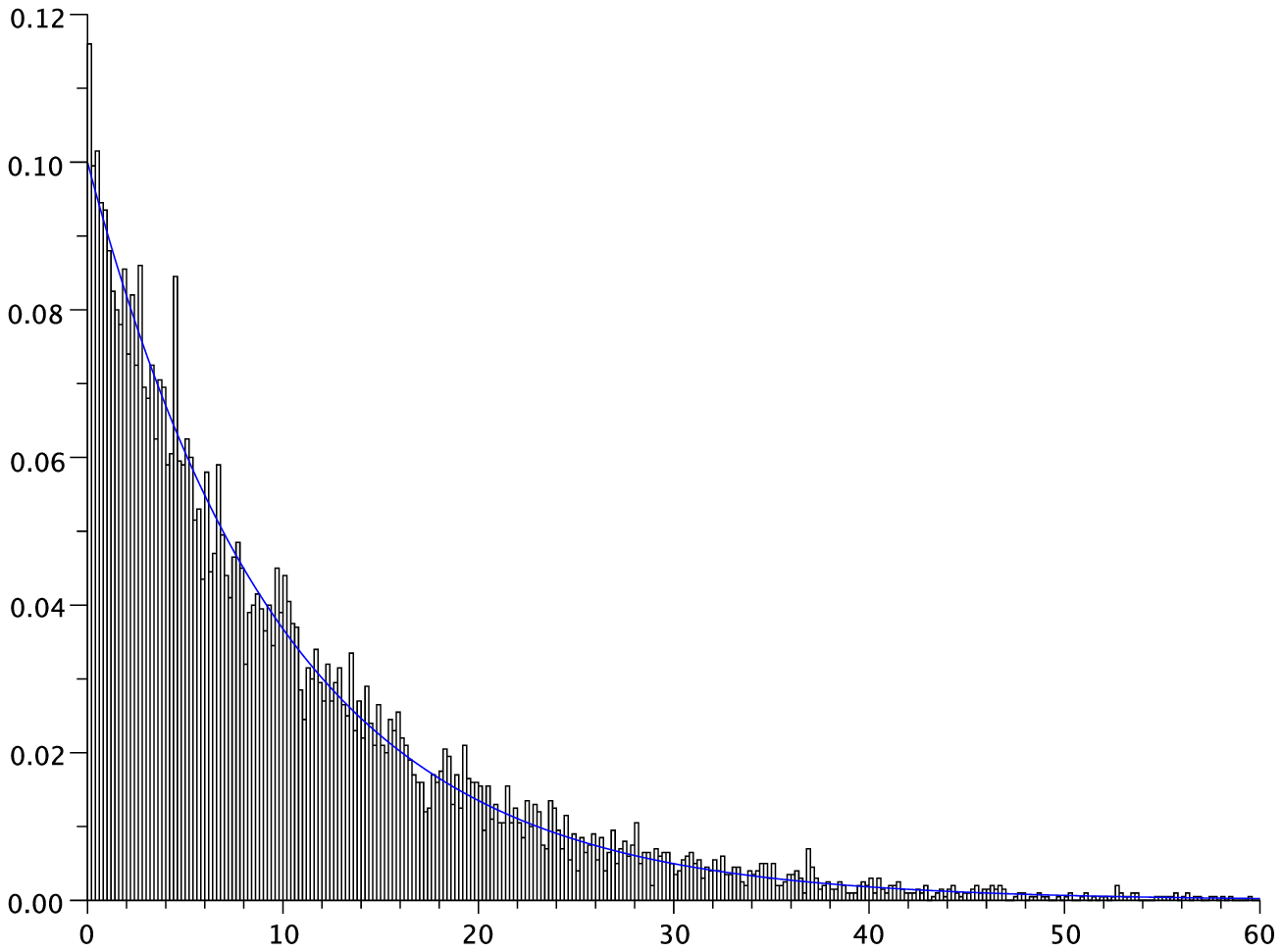}
}
\caption{Distribution of $S_0^\epsilon$, the first exit time from a
  macro-state ($m=20$). Left: large $\epsilon=1$. Right: small
  $\epsilon=10^{-3}$. 
\label{hist_3}}
\end{figure}
%
%

\begin{figure}[h]
\centerline{
\includegraphics[scale=0.4]{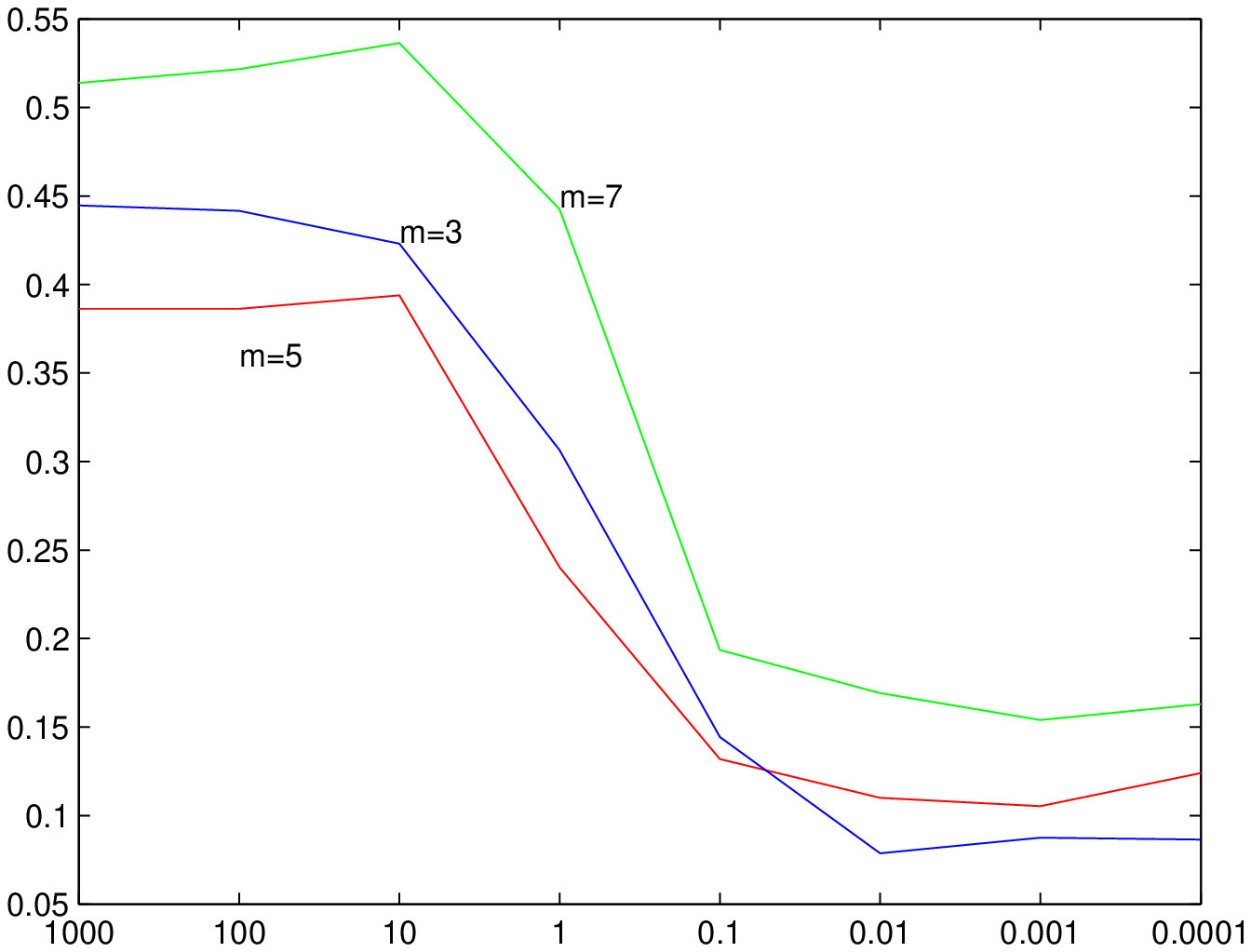}
\includegraphics[scale=0.4]{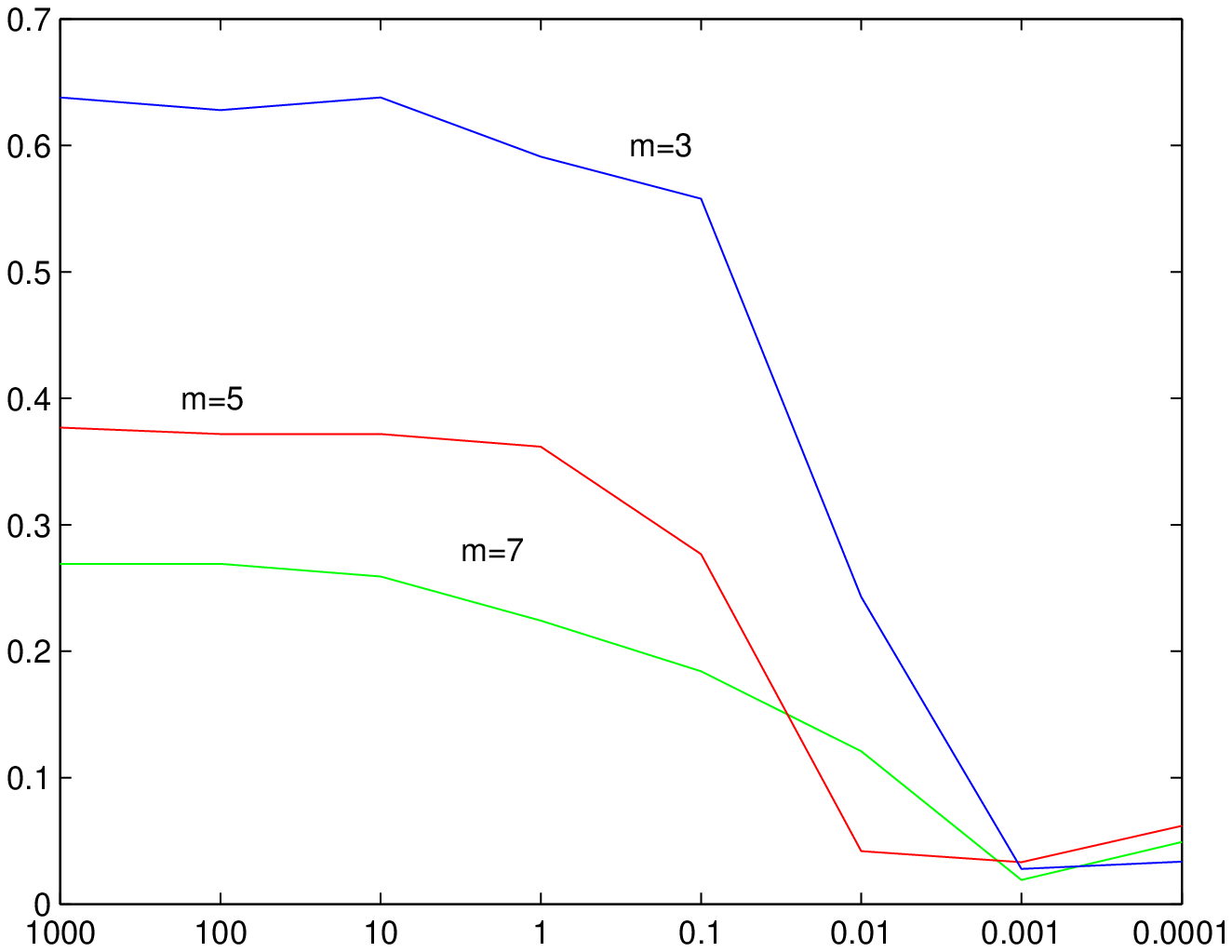}
}
\caption{$L^1$ error~\eqref{eq:err_L1} (left) and 
discrepancy~\eqref{eq:discrepancy} (right)
on the distribution of $S^\epsilon_0$ as a function of $\epsilon$.
\label{faux}}
\end{figure}
%
%

\medskip

We next monitor how the distribution of $S^\epsilon_0$ behaves when we
vary the initial condition. For this test, we work with
$m=5$. Figures~\ref{CI_moy} and~\ref{CI_var} show the empirical
expectation and variance for different initial positions and for
different values of $\epsilon$. We notice that, for an initial condition
which is at the middle of the macro-state, the convergence with respect to
$\epsilon$ is slower than for
the initial conditions which are at the boundaries of a macro-state. 
This difference is due to the diffusion phenomenon which occurs inside
each macro-state as a result of the transition to the nearest neighbours.

\begin{figure}[htbp]
\begin{center}
\psfrag{0.0}{\tiny $0\atop 0$}
\psfrag{1.0}{\tiny $1 \atop 0$}
\psfrag{2.0}{\tiny $2\atop 0$}
\psfrag{3.0}{\tiny $3\atop 0$}
\psfrag{4.0}{\tiny $4 \atop 0$}
\psfrag{0.1}{\tiny $0\atop 1$}
\psfrag{1.1}{\tiny $1\atop 1$}
\psfrag{2.1}{\tiny $2 \atop 1$}
\psfrag{3.1}{\tiny $3\atop 1$}
\psfrag{4.1}{\tiny $4\atop 1$}
\psfrag{moyenne de S0}{}
\psfrag{'moyenne_ene.txt'}{\tiny \hspace{0.7cm} mean}
\psfrag{position initiale, pour epsilon=10}{\tiny \hspace{0.7cm} initial condition}
\psfrag{3}{}
\psfrag{,1,10}{}
\psfrag{-3}{}
\includegraphics[scale=0.75]{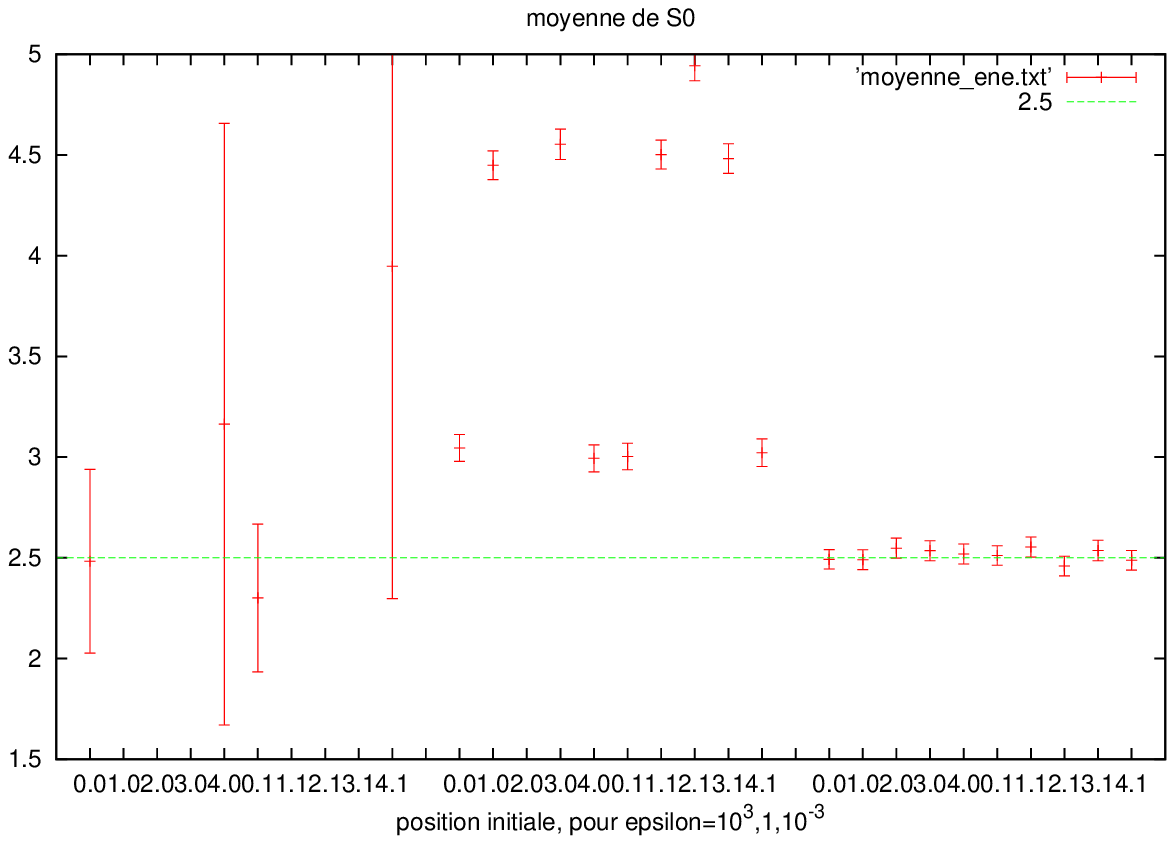}
\end{center}
\caption[pour eviter erreur compilo]{Empirical expectation of $S^\epsilon_0$ for different initial
  conditions and for $\epsilon=10^3$ (left), $\epsilon = 1$ (center) and
  $\epsilon=10^{-3}$ (right). Initial conditions are shown on the x-axis
  in the format $\left( \begin{array}{c} X_0 \\ Z_0 \end{array} \right)
  \in \left( \begin{array}{c} M \\ \{0,1\} \end{array} \right)$, with
  $M=\{0,1,2,3,4\}$.
\label{CI_moy}}
\end{figure}

\begin{figure}[htbp]
\begin{center}
\psfrag{0.0}{\tiny $0\atop 0$}
\psfrag{1.0}{\tiny $1 \atop 0$}
\psfrag{2.0}{\tiny $2\atop 0$}
\psfrag{3.0}{\tiny $3\atop 0$}
\psfrag{4.0}{\tiny $4 \atop 0$}
\psfrag{0.1}{\tiny $0\atop 1$}
\psfrag{1.1}{\tiny $1\atop 1$}
\psfrag{2.1}{\tiny $2 \atop 1$}
\psfrag{3.1}{\tiny $3\atop 1$}
\psfrag{4.1}{\tiny $4\atop 1$}
\psfrag{variance de S0}{}
\psfrag{'variance_ene.txt'}{\tiny \hspace{0.3cm} variance}
\psfrag{position initiale, pour epsilon=10}{\tiny \hspace{0.7cm} initial condition}
\psfrag{3}{}
\psfrag{,1,10}{}
\psfrag{-3}{}
\includegraphics[scale=0.75]{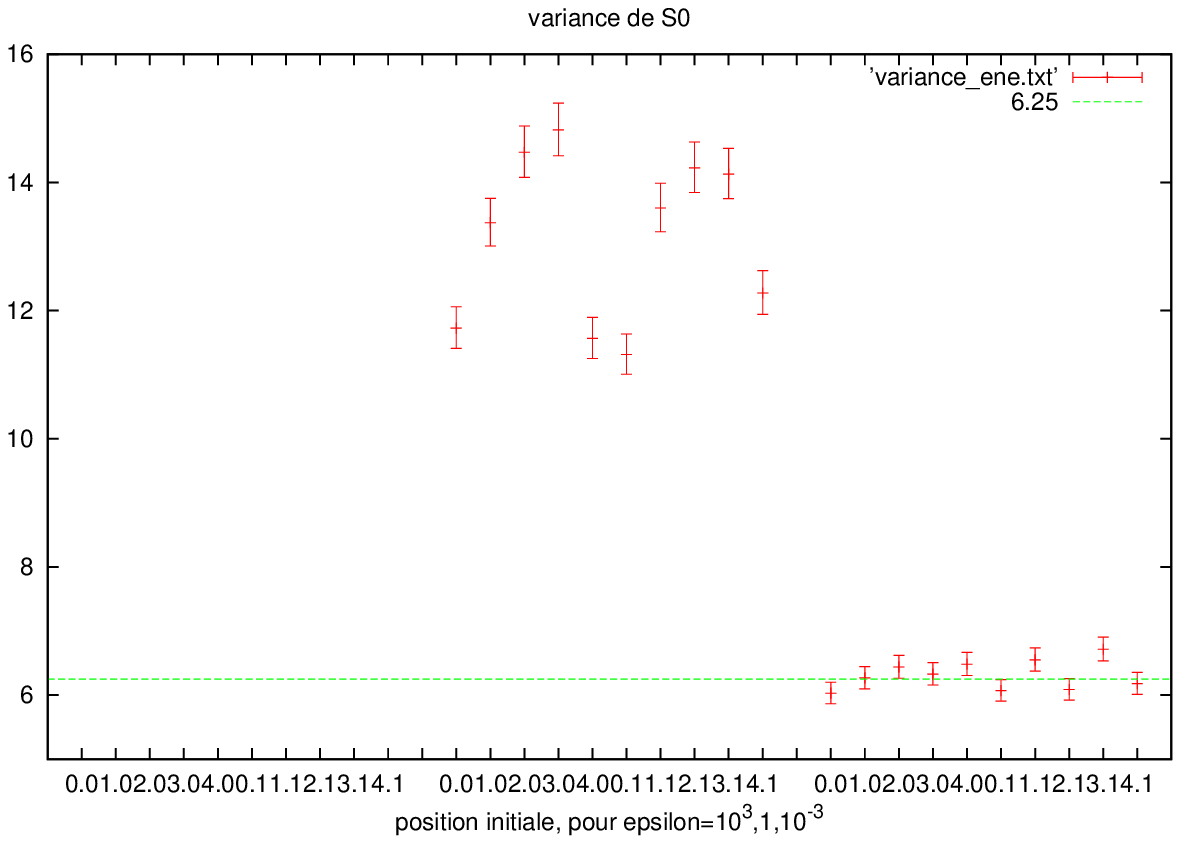}
\end{center}
\caption{Empirical variance of $S^\epsilon_0$ for different initial
  conditions and different values of $\epsilon$, with the same
  convention as on Fig.~\ref{CI_moy} (results for $\epsilon=
  10^3$ do not fit in the chosen y-range). 
\label{CI_var}}
\end{figure}

To better understand the behavior of the system for large values of $m$,
we have simulated our model with $m=20$. 
We show on Figs.~\ref{m20_moy} and \ref{m20_var}
the empirical expectation and variance of $S^\epsilon_0$ for two
different initial conditions, one on the boundary ($Y_0=(0,0)$) and the
other in the middle of the macro-state ($Y_0=(10,1)$). On
Fig.~\ref{disc_20_00}, 
we show the
convergence of the distribution of $S_0^\epsilon$ to its limit for these two initial conditions.


We clearly see that the convergence is slower and the error margins are
larger (for the same number of 
Monte-Carlo realizations) than when we chose smaller values of $m$
(compare for example Fig.~\ref{m20_moy} with
Fig.~\ref{moyenne_carryon} or Fig.~\ref{disc_20_00} with
Fig.~\ref{faux}). The system indeed takes more time in a given
macro-state before reaching its boundary and possibly jumping.
 
\begin{figure}[htbp]
\begin{center}
\psfrag{moyenne de S0}{}
\psfrag{'moyenne_ene.txt'}{\tiny \hspace{0.7cm} mean}
\psfrag{epsilon, pour position initiale 0.0 et 10,1}{\tiny
  \hspace{-2.5cm} epsilon (initial condition $(0,0)$) \hspace{0.7cm} epsilon (initial condition $(10,1)$)}
\includegraphics[scale=0.75]{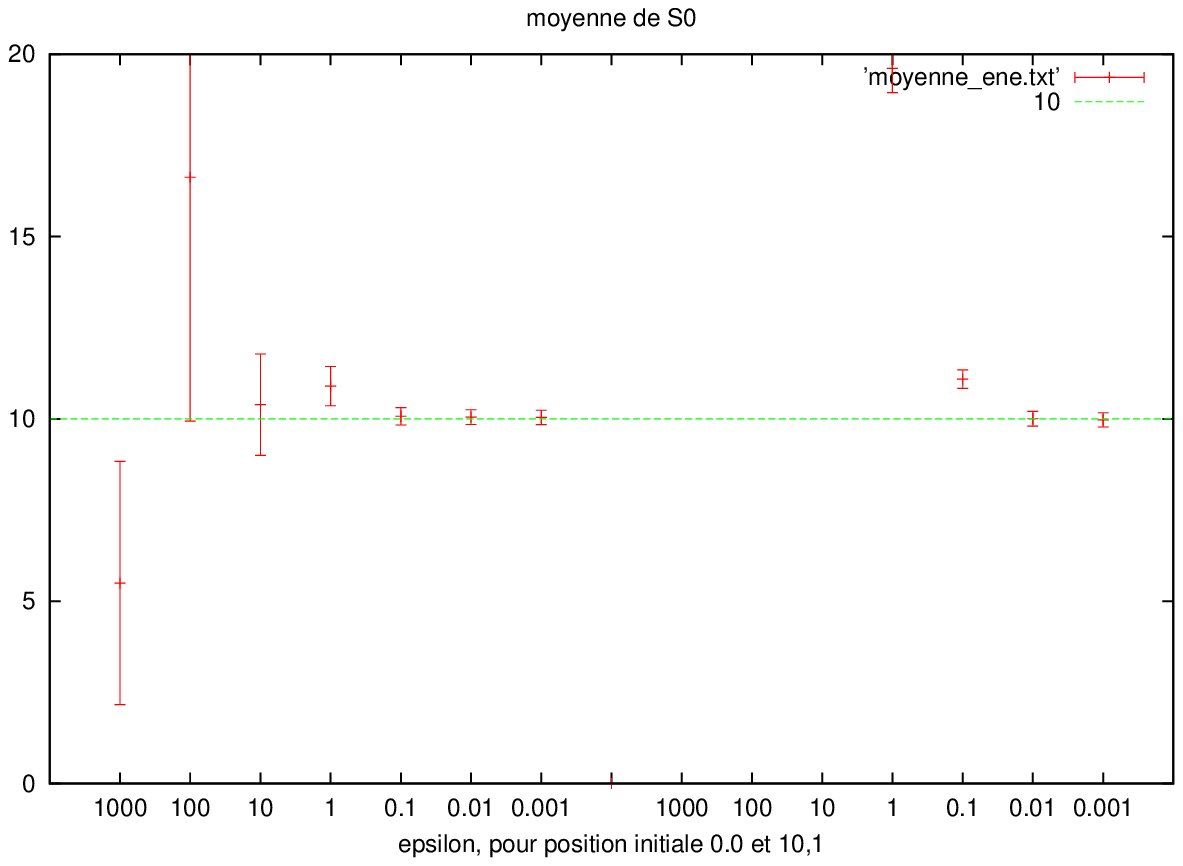}
\end{center}
\caption{Empirical expectation of $S^\epsilon_0$ for $m=20$ and two
  different initial conditions: $Y_0=(0,0)$ (left) and $Y_0=(10,1)$ (right). 
\label{m20_moy}}
\end{figure}

\begin{figure}[htbp]
\begin{center}
\psfrag{variance de S0}{}
\psfrag{'variance_ene.txt'}{\tiny \hspace{0.3cm} variance}
\psfrag{epsilon, pour position initiale 0.0 et 10,1}{\tiny \hspace{-2.5cm} epsilon (initial condition $(0,0)$) \hspace{0.7cm} epsilon (initial condition $(10,1)$)}
\includegraphics[scale=0.75]{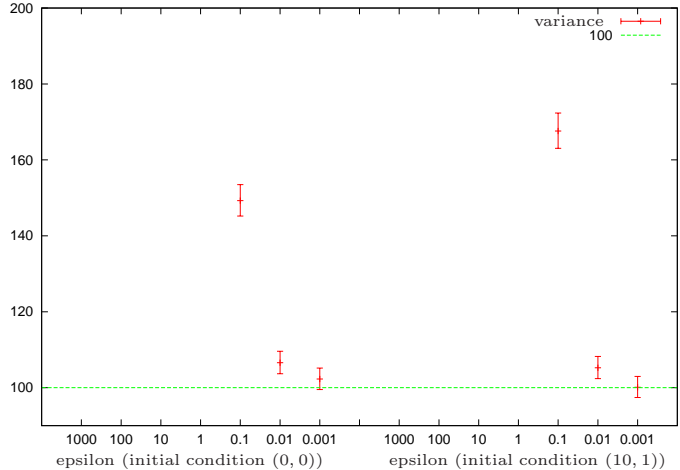}
\end{center}
\caption{Empirical variance of $S^\epsilon_0$ for $m=20$ and two
  different initial conditions: $Y_0=(0,0)$ (left) and $Y_0=(10,1)$ (right). 
\label{m20_var}}
\end{figure}


\begin{figure}[h]
\centerline{
\includegraphics[scale=0.3]{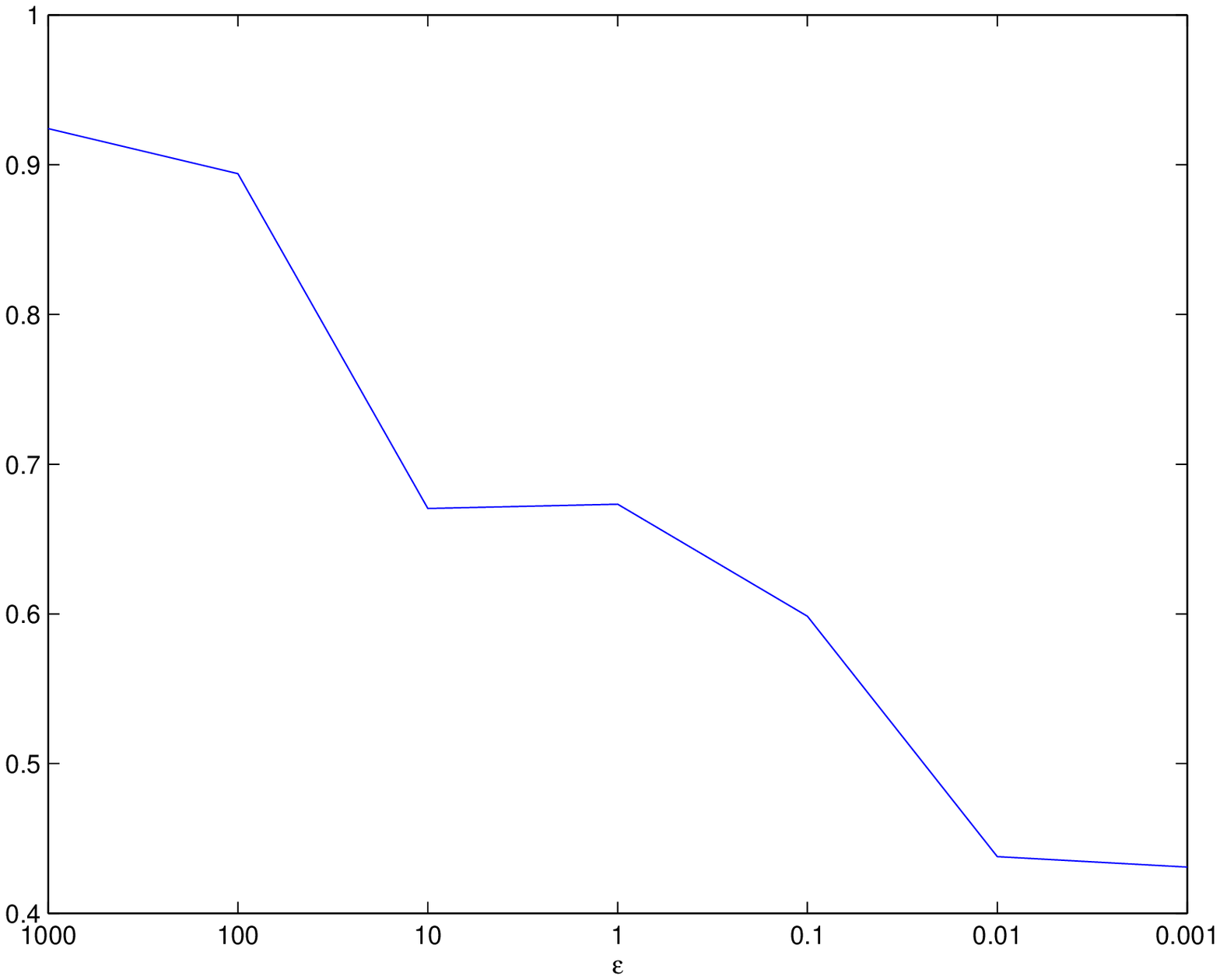}
\includegraphics[scale=0.3]{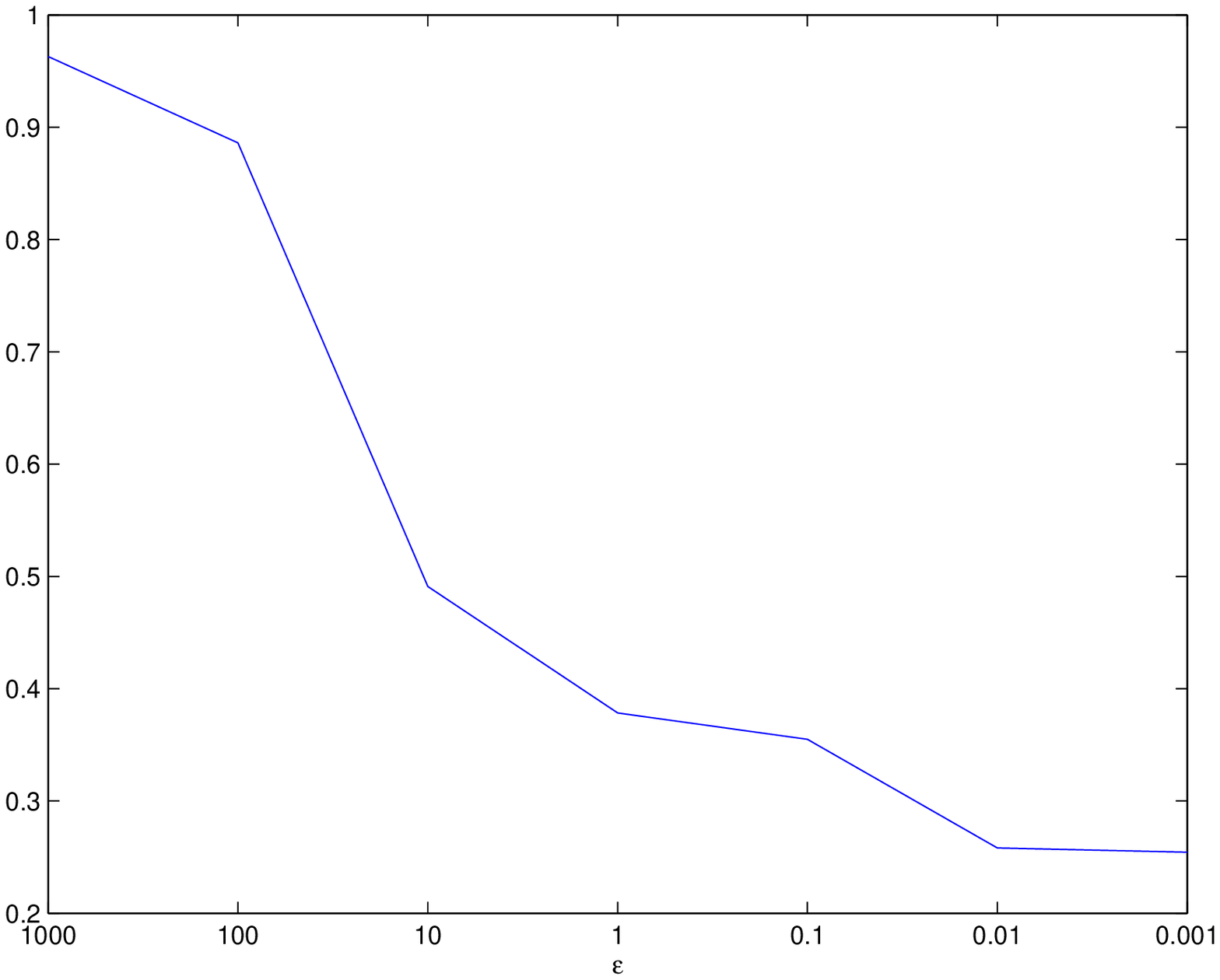}
}
\caption{$L^1$ error~\eqref{eq:err_L1} on the distribution of
  $S^\epsilon_0$ for $m=20$.
Left: initial condition $Y_0=(0,0)$.
Right: initial condition $Y_0=(10,1)$.
\label{disc_20_00}}
\end{figure}
%
%

\medskip

To conclude this numerical illustration, we have monitored the
distribution of $S_1^\epsilon$, the exit time from the second macro
state, and compared it with that of $S_0^\epsilon$, the exit time from
the first macro-state.
We observe (results not shown) that $S^\epsilon_1$ has the
same asymptotic behaviour as $S^\epsilon_0$, a fact which is in agreement with
the theoretical predictions.





\begin{remark}
The parameters of the numerical simulations reported here have been
chosen so that the limit dynamics (at $\epsilon = 0$) is an inaccurate
approximation of the reference dynamics when $\epsilon$ is large (say
$\epsilon \geq 1$). 

There are actually cases when the limit dynamics is an accurate
approximation of the reference dynamics, even if $\epsilon$ is not
small. For example, consider the case where, 
for a given macro-state (say $Z=0$), the
  transitions from each micro-state of this macro-state to any
  micro-state of the other macro-state ($Z=1$) share the same
  frequency. In the case of the symmetric model considered in
  Section~\ref{section_1_2}, the homogeneity condition means that 
$$
\sum_{x'\in M} C\left(x,x'\right) = \text{Cte independent of $x$}.
$$ 
In this case, the macroscopic dynamic is decoupled from the 
microscopic variable, as can be seen from~\eqref{eqfonda}, and of course
does not depend on $\epsilon$.
%
\end{remark}


\section{A particle in a potential energy landscape with infinitely many macro-states}
\label{infinite_detats}

In Section~\ref{modele_deux_etats}, we have studied the dynamics of a
particle in a potential energy with two macro-states. We now turn to the
system composed of a particle in a potential
energy with infinitely many macro-states. We establish a convergence
result on the dynamics of a slow quantity of interest in
Section~\ref{sym}, before turning to numerical illustrations in
Section~\ref{2}. 

\subsection{Presentation of the model and main result}
\label{sym}

As mentioned above, we consider here the dynamics of a particle in a
potential energy with infinitely many macro-states. As in
Section~\ref{section_1_1}, the state of the particle is described by
$\overline{Y_t^\epsilon}
=\left(\overline{X_t^\epsilon},\overline{Z_t^\epsilon}\right)$, which takes
its values in $M \times \mathbb Z$, where again $\overline{X_t^\epsilon}
\in M = \{1,\dots,m\}$ is the label of the micro-state in which the
particle is. The variable $\overline{Z_t^\epsilon}$ is the label of the
macro-state in which the particle is at time $t$, and it now takes any
value of $\ZZ$. 

For simplicity, we assume that the dynamics within each macro-state is
similar. We also 
restrict the transitions from one macro-state to its two neighbors. The
transition from $z$ to $z+1$ may have different properties than the
transition from $z$ to $z-1$ (thus creating a macroscopic drift in the
dynamics). We also assume that the system is 
macroscopically homogeneous, in the sense that properties are translation
invariant with respect to $z$. Under these assumptions, a typical
transition intensity for the process
$\left(\overline{Y_t^\epsilon}\right)_t$ is given by
\begin{equation}
\label{eq:def_Q_eps_many}
\begin{array}{rcl}
\forall z \in \ZZ, \quad
\overline{Q}^\epsilon\left(\left(x,z\right),\left(x',z\right)\right)
&=&
Q\left(x,x'\right),
\\
\forall z \in \ZZ, \quad
\overline{Q}^\epsilon\left(\left(x,z\right),\left(x',z+1\right)\right)
&=&
\epsilon C_r\left(x,x'\right),
\\ 
\forall z \in \ZZ, \quad
\overline{Q}^\epsilon\left(\left(x,z\right),\left(x',z-1\right)\right)
&=&
\epsilon C_l\left(x,x'\right),
\\ 
\forall z \in \ZZ, \quad
\overline{Q}^\epsilon\left(\left(x,z\right),\left(x',z'\right)\right)
&=&
0 \ \mbox{if $z'\neq z$, $z+1$ or $z-1$.}
\end{array}
\end{equation}

We again assume that the matrix $Q$ is irreducible (see~\eqref{eq:hyp})
and introduce its 
unique invariant measure $\pi$. The average of the jump frequency according
to the invariant measure reads
\begin{equation}
\label{eq:lambda_inf}
\lambda_l = \sum_{x,x'\in M} C_l\left(x,x'\right)\pi\left(x\right),
\quad
\lambda_r = \sum_{x,x'\in M} C_r\left(x,x'\right)\pi\left(x\right).
\end{equation}
We introduce the generator $L$ defined by: for any bounded function
$\phi$ on $\ZZ$, 
\begin{equation}
\label{eq:def_L_many}
L\phi(z)=\lambda_l \phi(z-1) + \lambda_r \phi(z+1) - \left(\lambda_r+\lambda_l\right) \phi(z),
\end{equation}
which is the generator of a jump process $(Z_t)_{t\geq 0}$ on $\ZZ$,
with jumps at times
defined by a Poisson process of parameter $\lambda_l+\lambda_r$. When
the process jumps, it jumps to the right (resp. to the left) with
probability $\dps \frac{\lambda_r}{\lambda_r+\lambda_l}$
(resp. $\dps \frac{\lambda_l}{\lambda_r+\lambda_l}$ ). 


The main result of this section is the following:
\begin{theorem}
\label{th2} 
Assume that the matrix $Q$ is irreducible. 
Consider the rescaled-in-time process
$Y^\epsilon_t=(X^\epsilon_t,Z^\epsilon_t) 
= \overline{Y}^\epsilon_{t/\epsilon}$ with initial condition 
$Y_0=(X_0,Z_0)$ independent of $\epsilon$.
We denote by $\mathcal P^\epsilon$ the distribution of the process
$\left(Z_t^\epsilon\right)_t$ and by $\mathcal P$ the distribution of the
process starting from the initial condition $Z_0$ and having as
generator the operator $L$ defined by~\eqref{eq:def_L_many}. Then 
$$
\mathcal P^{\epsilon} \Rightarrow \mathcal P \mbox{ as $\epsilon$ goes
  to $0$}.
$$
\end{theorem}

The proof of this result follows the same steps as that of
Theorem~\ref{th1}, up to the fact that the process $Z^\epsilon$ is no
longer bounded. To circumvent this difficulty, we need to work with
an {\em arbitrary} bounded function of $Z^\epsilon$, in contrast to the proof
of Theorem~\ref{th1}, where it is sufficient to directly work with
$Z^\epsilon$. 

We briefly sketch the proof of Theorem~\ref{th2}. 
The generator $L^\epsilon$ of $Y^\epsilon_t$ reads, for a bounded function $\phi$,
\begin{multline*}
L^\epsilon \phi\left(x,z\right)
= 
\sum_{x'\in M} \epsilon^{-1} Q\left(x,x'\right)
\left(\phi\left(x',z\right)-\phi\left(x,z\right)\right)
\\
+ 
\sum_{x'\in M} C_l\left(x,x'\right) \left( \phi\left(x',z-1\right) 
- \phi\left(x,z\right)\right)
\\
+ 
\sum_{x'\in M} C_r\left(x,x'\right) \left( \phi\left(x',z+1\right) - \phi\left(x,z\right)\right).
\end{multline*}
For a function $\phi\left(x,z\right)=F(z)$ which only depends on the
macroscopic variable (where $F$ is a bounded function on $\mathbb Z$),
we have 
\begin{multline*}
(L^\epsilon F)\left(x,z\right)
= 
\sum_{x'\in M} C_l\left(x,x'\right)
\left(F\left(z-1\right)-F\left(z\right)\right)
\\
+
\sum_{x'\in M} C_r\left(x,x'\right) \left( F\left(z+1\right) -
  F\left(z\right)\right).
\end{multline*}
Using Proposition~\ref{lemma1}, we know that the process 
\begin{equation}
\label{eqfonda2_pre} 
M^\epsilon_t
=
F\left(Z^\epsilon_t\right)- F\left(Z_0\right) - 
\int^t_0 (L^\epsilon F)\left(X^\epsilon_s,Z^\epsilon_s\right) \, ds
\end{equation}
is a $\mathcal F_t^\epsilon$-martingale. 
We now introduce
\begin{multline}
\label{G}
G(F)\left(x,z\right)
= 
\left(F\left(z-1\right)-F\left(z\right)\right)
\sum_{x'\in M} \left( C_l\left(x,x'\right) - \lambda_l \right)
\\
+
\left( F\left(z+1\right) - F\left(z\right)\right)
\sum_{x'\in M} \left( C_r\left(x,x'\right) - \lambda_r \right),
\end{multline}
so that
$$
(L^\epsilon F)\left(x,z\right)
=
G(F)\left(x,z\right)
+
LF(z)
$$
where $L$ is defined by~\eqref{eq:def_L_many}. We then
recast~\eqref{eqfonda2_pre} as
\begin{equation}
F\left(Z^\epsilon_t\right)
=
F\left(Z_0\right)
+
\int^t_0 G\left(F\right)\left(Y_s^{\epsilon}\right)ds 
+  
\int^t_0 LF\left(Z_s^{\epsilon}\right)ds + M^\epsilon_t.
\label{eqfonda2}
\end{equation}  
We are now left with passing to the limit $\epsilon \to 0$
in~\eqref{eqfonda2}.

\medskip

Consider first the second term of the right-hand
side of~\eqref{eqfonda2}. We have the following result (compare with
Proposition~\ref{prop1}): 

\begin{proposition}
For any bounded function $F$ defined on $\ZZ$ and any $t\geq 0$, under
the assumptions of Theorem~\ref{th2}, we have 
$$ 
\EE \left[ \left( 
\int_0^t G\left(F\right)\left(Y_s^\epsilon\right) \, ds 
\right)^2 \right]
\longrightarrow 0 \quad \text{as $\epsilon \to 0$}, 
$$ 
where $G(F)$ is defined by~\eqref{G}. 
\end{proposition}

\begin{proof}
The proof follows the same steps as that of Proposition~\ref{prop1}. Fix
$z \in \mathbb Z$ and consider the function 
$x \in M \mapsto G_z(x)=G(F)(x,z)$, that we identify with a vector in
$\RR^m$, denoted $G_z$. Using~\eqref{eq:lambda_inf}, we observe that
$\pi^T G_z=0$. We then deduce that, for any $y \in \RR^m$ such that $y^T
(Q-\Delta) = 0$ (where $\Delta$ has been defined in
Remark~\ref{rem:convention2}), we have $y^T G_z = 0$. We then infer from 
Lemma~\ref{lemme1} that there exists $u_z \in \RR^m$ such that $(Q -
\Delta) \, u_z=G_z$. Introducing the function $u(x,z)=u_z(x)$, we easily
check that $\overline{L}^0u=G(F)$. The rest of the proof is identical to that of
Proposition~\ref{prop1}.
\end{proof}

For the other terms of~\eqref{eqfonda2}, the proof follows exactly the
same steps as in the proof of Theorem~\ref{th1}. We hence obtain that
the weak limit $Z$ of $\left(Z^\epsilon\right)$ satisfies that, for
every bounded function $F$ on $\ZZ$, there exists a martingale $M^F$
such that 
\begin{equation} 
\label{eqfonda3}
F\left(Z_t\right) 
= 
F\left(Z_0\right) + \int^t_0 LF\left(Z_s\right)ds + M^F_t.
\end{equation}
Using Lemma~\ref{appendixG}, we conclude that $Z$ is a jump process of
generator $L$ defined by~\eqref{eq:def_L_many}.

\begin{remark}
We refer to~\cite{these_salma} for the study of the limit process
introduced in Theorem~\ref{th2}, after a rescaling both {\em in time
and space}. We show there that it converges to a Brownian motion (up to a
multiplicative constant).
\end{remark}

\subsection{Numerical illustration}
\label{2}

We have simulated the model described in Section~\ref{sym}, with the choices
$$
Q = \left(
\begin{array}{ccccc}
0 & q & & &  \\
q & 0 & q & & \\
\ddots & \ddots & \ddots & \\
 & &q &0 & q \\
 & & & q & 0 \\ 
\end{array}
\right),
$$
$$
C_l = \left(
\begin{array}{ccccc}
0 & \cdots&  &0  &  c_l  \\
0 & \cdots &  &  & 0\\
\vdots & & & & \vdots \\
0 & 0 & \cdots&  & 0 \\ 
\end{array}
\right)
\quad \text{and} \quad
C_r = \left(
\begin{array}{ccccc}
0 & \cdots&  &0  &  0  \\
0 & \cdots &  &  & 0\\
\vdots & & & & \vdots \\
c_r & 0 & \cdots&  & 0 \\ 
\end{array}
\right)
$$
with $q=1$, $c_r=2$, $c_l=1$, $m=5$ and the initial condition
$Y_0=(0,0)$ (similar results are obtained for other initial
conditions). The parameters $\lambda_r$ and $\lambda_l$ of the 
macroscopic evolution are
$$
\lambda_l= \frac{c_l}{m}=\frac{1}{5} 
\quad \text{and} \quad
\lambda_r=\frac{c_r}{m}=\frac{2}{5}.
$$
We first monitor the convergence of the distribution of $S^\epsilon_0$,
the exit time from the first well. On
Fig.~\ref{Z_moyenne_variance_S_0}, we show its empirical expectation and
variance. We see that they converge to their
asymptotic values as $\epsilon$ goes
to zero. This convergence is confirmed by the histogram representation
(on Fig.~\ref{hist_Z_S_0}), where we see a good agreement between the
discrete curve and the asymptotic curve for sufficiently small values of
$\epsilon$. Likewise, the $L^1$ error, also shown on Fig.~\ref{hist_Z_S_0},
indeed converges to zero. 

\begin{figure}[h]
\begin{center}
\psfrag{moyenne de S0}{}
\psfrag{'moyenne_ene.txt'}{\tiny \hspace{0.3cm} mean}
\psfrag{variance de S0}{}
\psfrag{'variance_ene.txt'}{\tiny variance}
\includegraphics[scale=0.49]{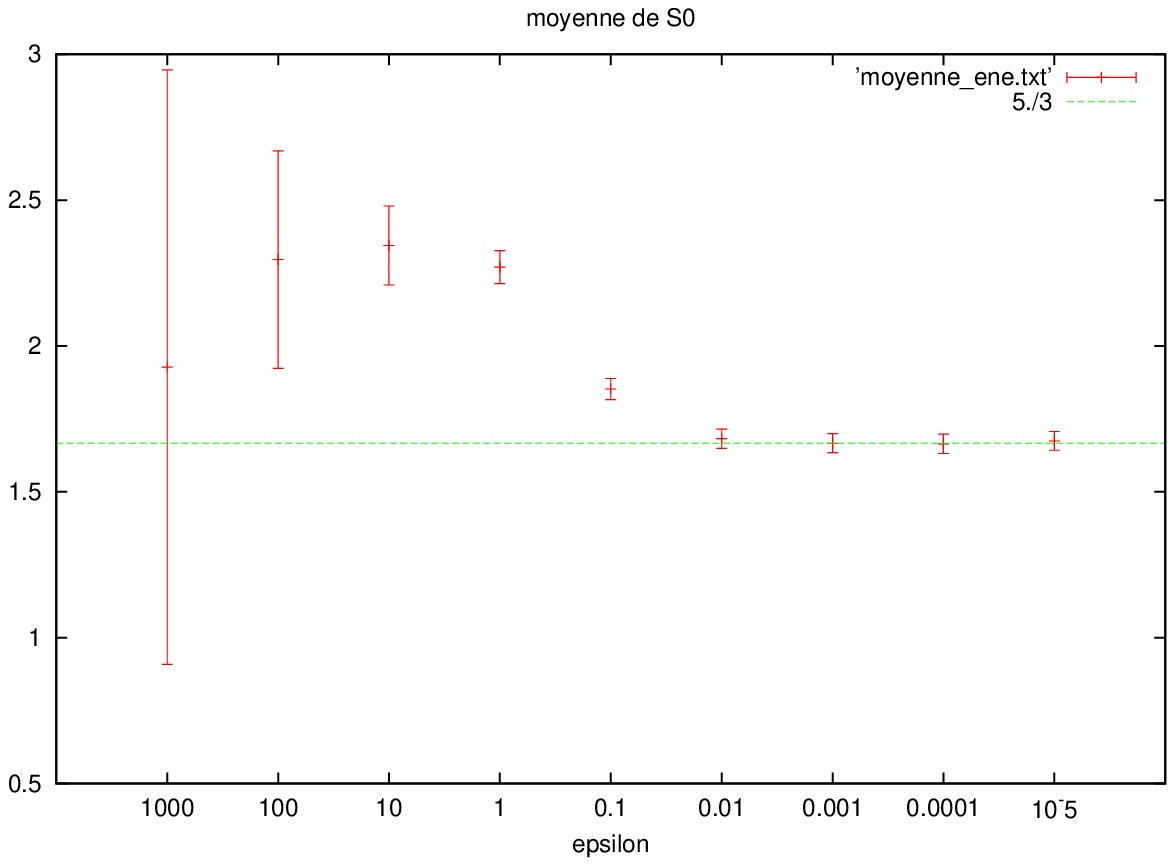}
\includegraphics[scale=0.49]{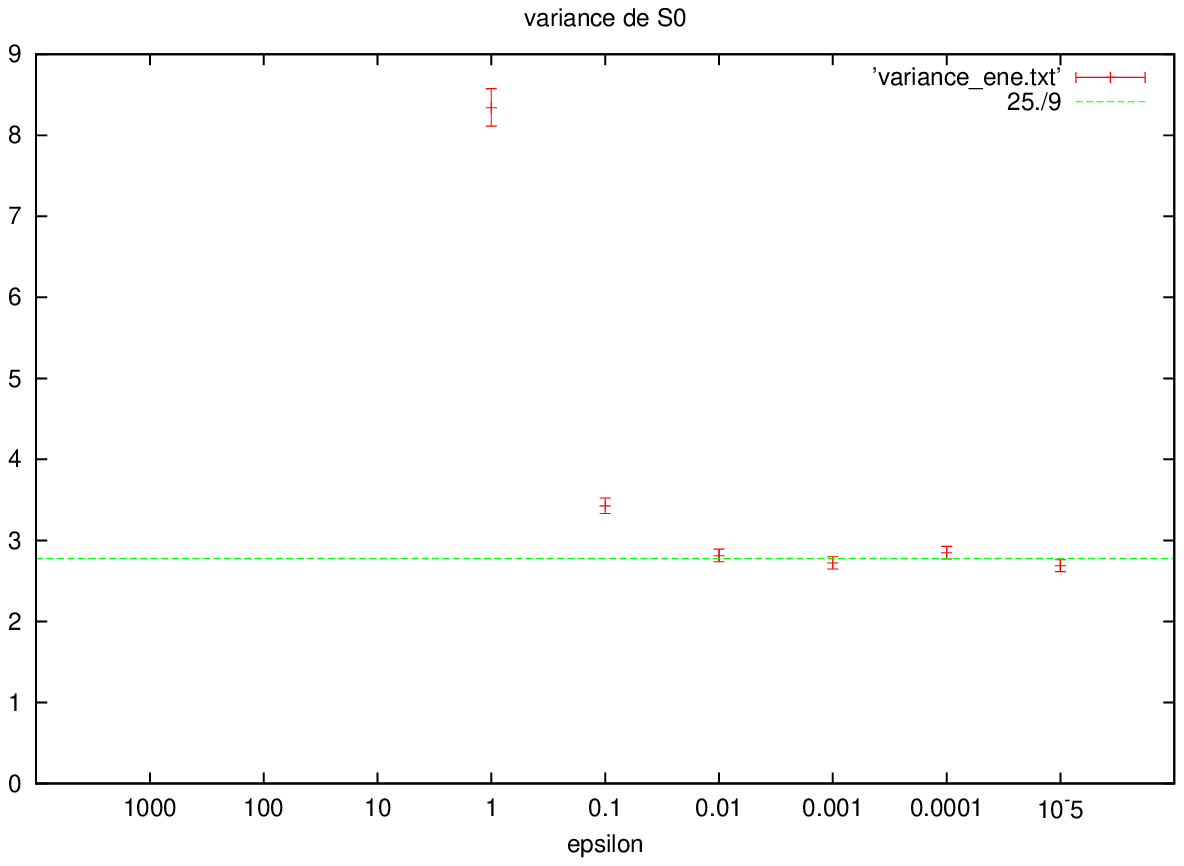}
\end{center}
\caption{Empirical expectation (left) and variance (right) of
  $S^\epsilon_0$ as a function of $\epsilon$.
\label{Z_moyenne_variance_S_0}}
\end{figure}

%


%
\begin{figure}[h!]
\centerline{ 
\includegraphics[scale=0.35]{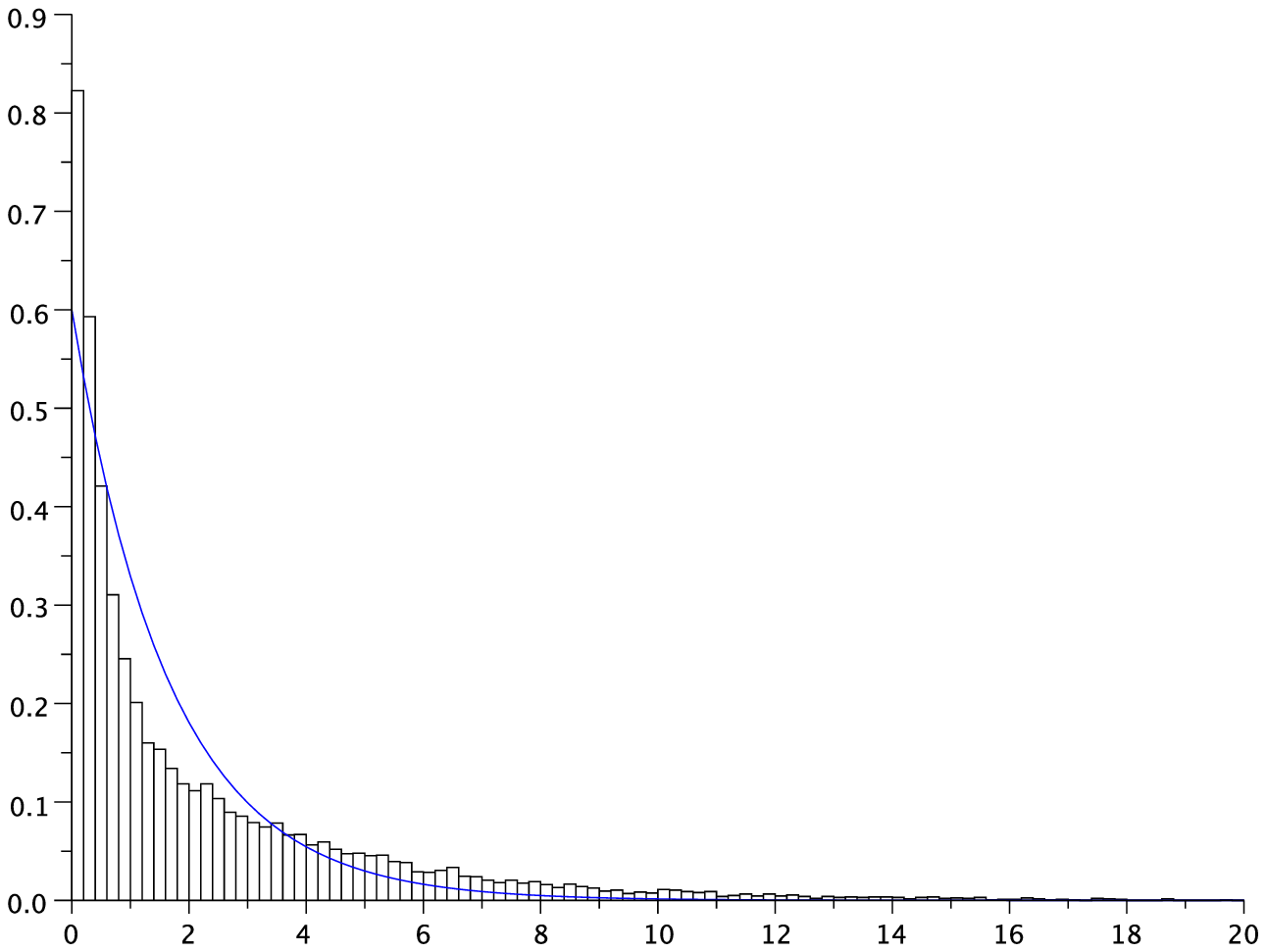}
\includegraphics[scale=0.35]{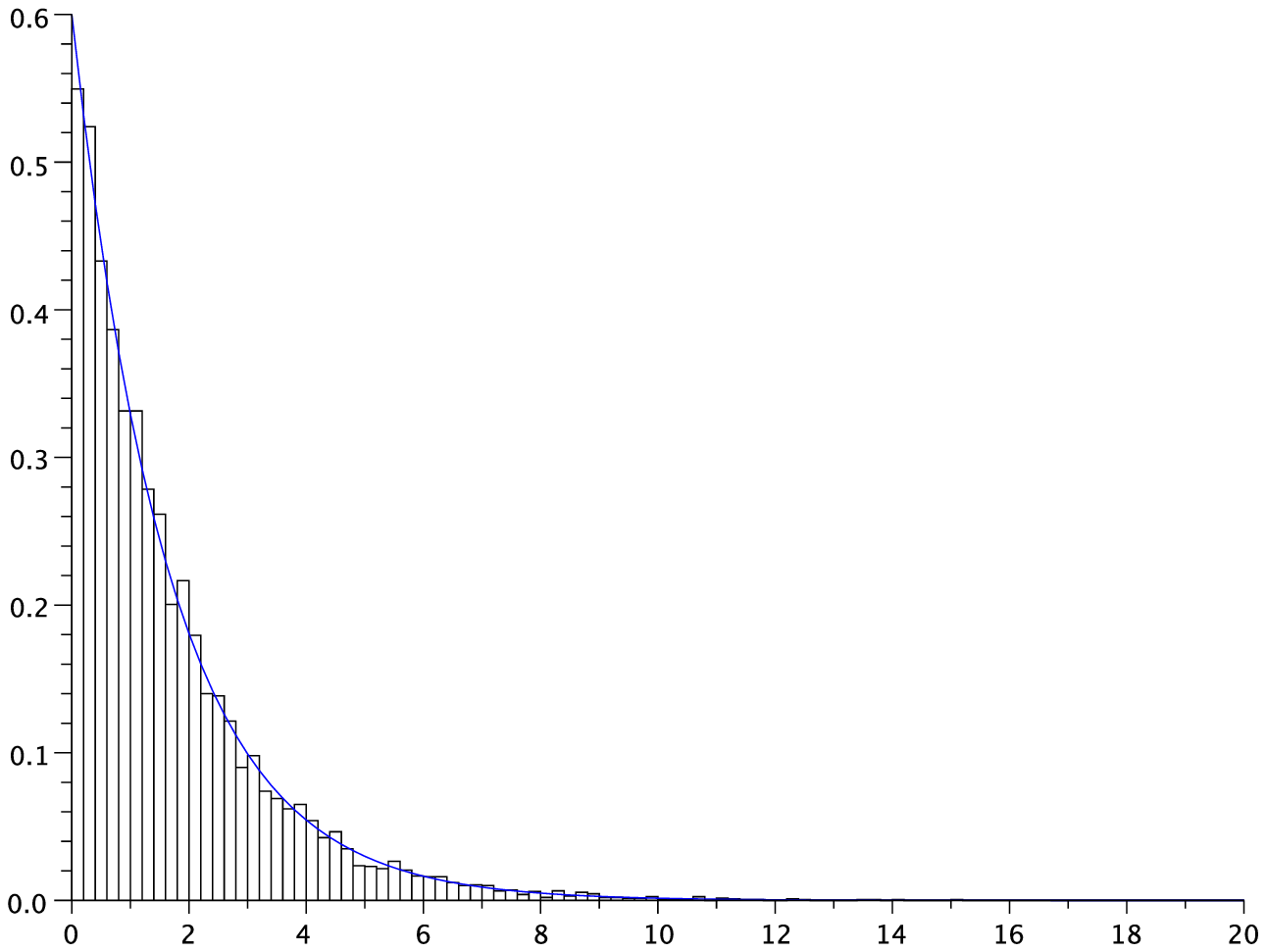}
\includegraphics[scale=0.25]{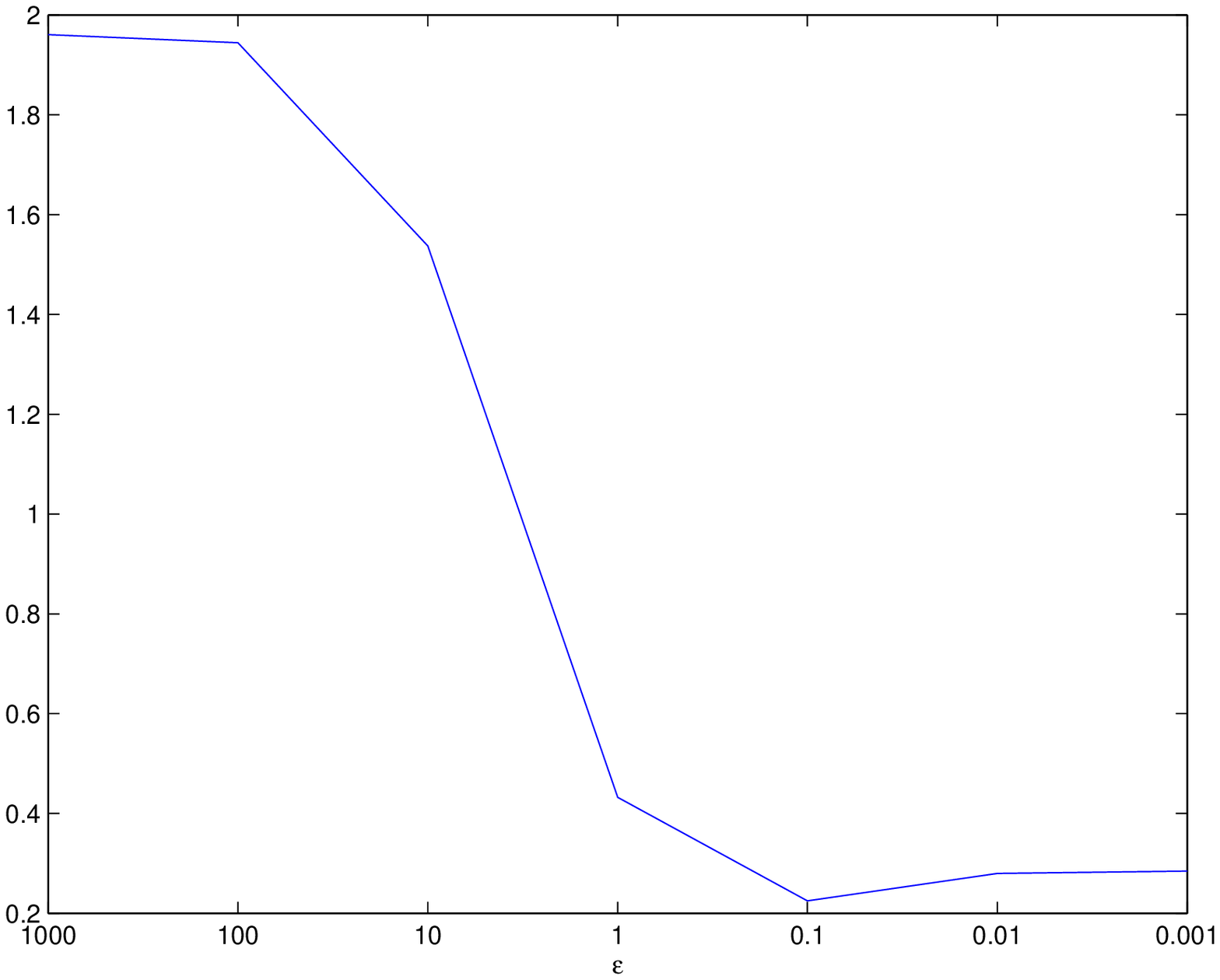}
}
\caption{Left and Center: Distribution of $S^\epsilon_0$, the first
  exit time from a macro-state (Left: large $\epsilon=1$. Center: small
  $\epsilon=10^{-3}$). 
Right: $L^1$ error~\eqref{eq:err_L1} on the distribution of
  $S^\epsilon_0$, as a function of $\epsilon$.
\label{hist_Z_S_0}}
\end{figure}


\medskip

We next study the distribution of the amplitude of the first jump of the
macroscopic variable $Z^\epsilon$, that is the distribution of the
random variable 
$$
\Delta Z^\epsilon := Z^\epsilon_{S^\epsilon_0}-Z_0.
$$ 
On Fig.~\ref{Z_moyenne_variance_z}, we show
the empirical expectation and variance of $\Delta
Z^\epsilon$, which are observed to converge to their asymptotic
values. Note that the limiting process $Z$, the generator of which is
the operator~\eqref{eq:def_L_many}, drifts to the right, since
$\lambda_r>\lambda_l$. We compute that
$$
\EE\bra{\Delta Z} = \PP\bra{\Delta Z=1} \times 1
+ \PP\bra{\Delta Z=-1} \times (-1)
=\frac{2/5}{3/5}-\frac{1/5}{3/5}=\frac{1}{3},
$$
and we indeed see on Fig.~\ref{Z_moyenne_variance_z} that 
$\dps \lim_{\epsilon \to 0} \EE\bra{\Delta Z^\epsilon} = \EE\bra{\Delta Z}$.
On Fig.~\ref{hist_Z_z}, we show the empirical distribution of $\Delta
Z^\epsilon$ for a small $\epsilon$, and we observe that
$$
\PP\bra{\Delta Z^\epsilon=1} 
\approx 
\PP\bra{\Delta Z=1} 
=\frac{2}{3},\quad 
\PP\bra{\Delta Z^\epsilon=-1}
\approx 
\PP\bra{\Delta Z=-1}
=\frac{1}{3}.
$$ 
We also check on Fig.~\ref{hist_Z_z}
that the $L^1$ error
between the distribution of $\Delta Z^\epsilon$ and that of $\Delta Z$
goes to 0 as $\epsilon$ goes to zero.

\begin{figure}[h]
\begin{center}
\psfrag{moyenne de Z}{}
\psfrag{'moyenne_ene.txt'}{\tiny \hspace{0.3cm} mean}
\psfrag{variance de Z}{}
\psfrag{'variance_ene.txt'}{\tiny variance}
\includegraphics[scale=0.49]{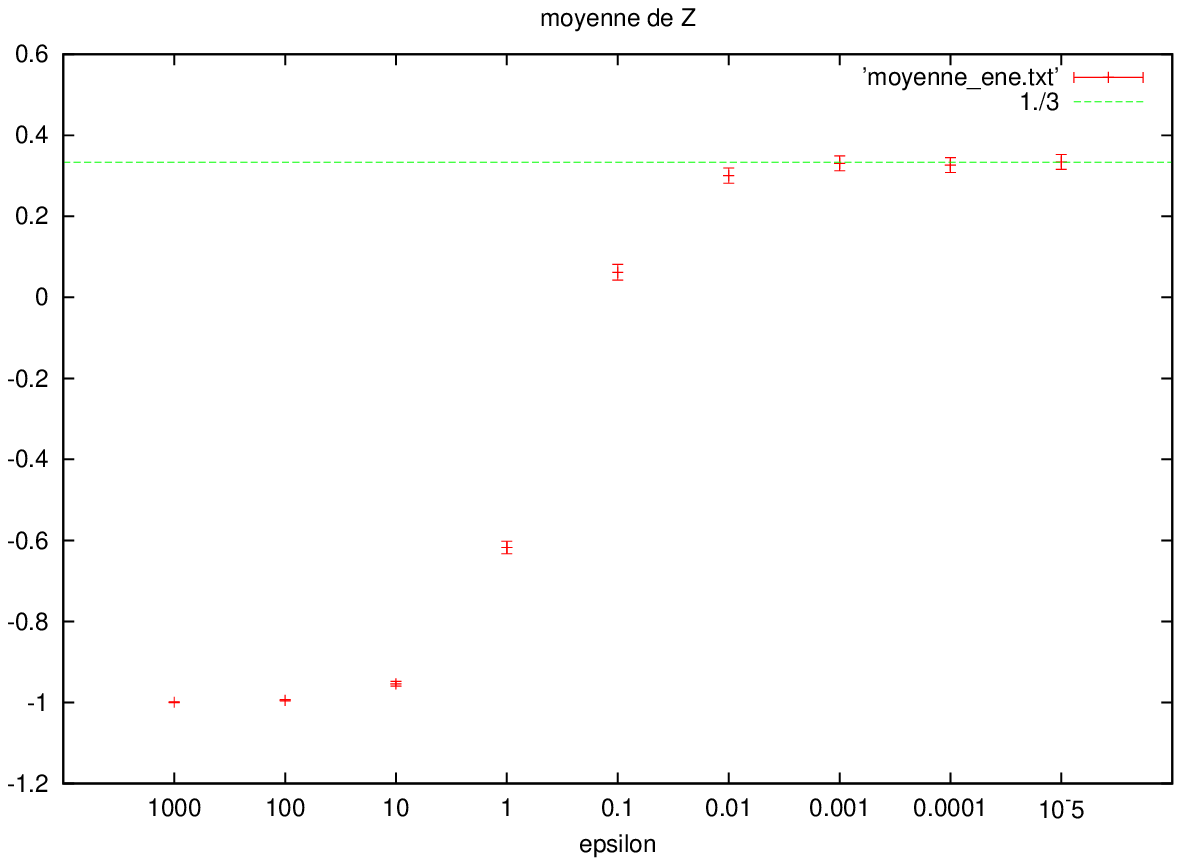}
\includegraphics[scale=0.49]{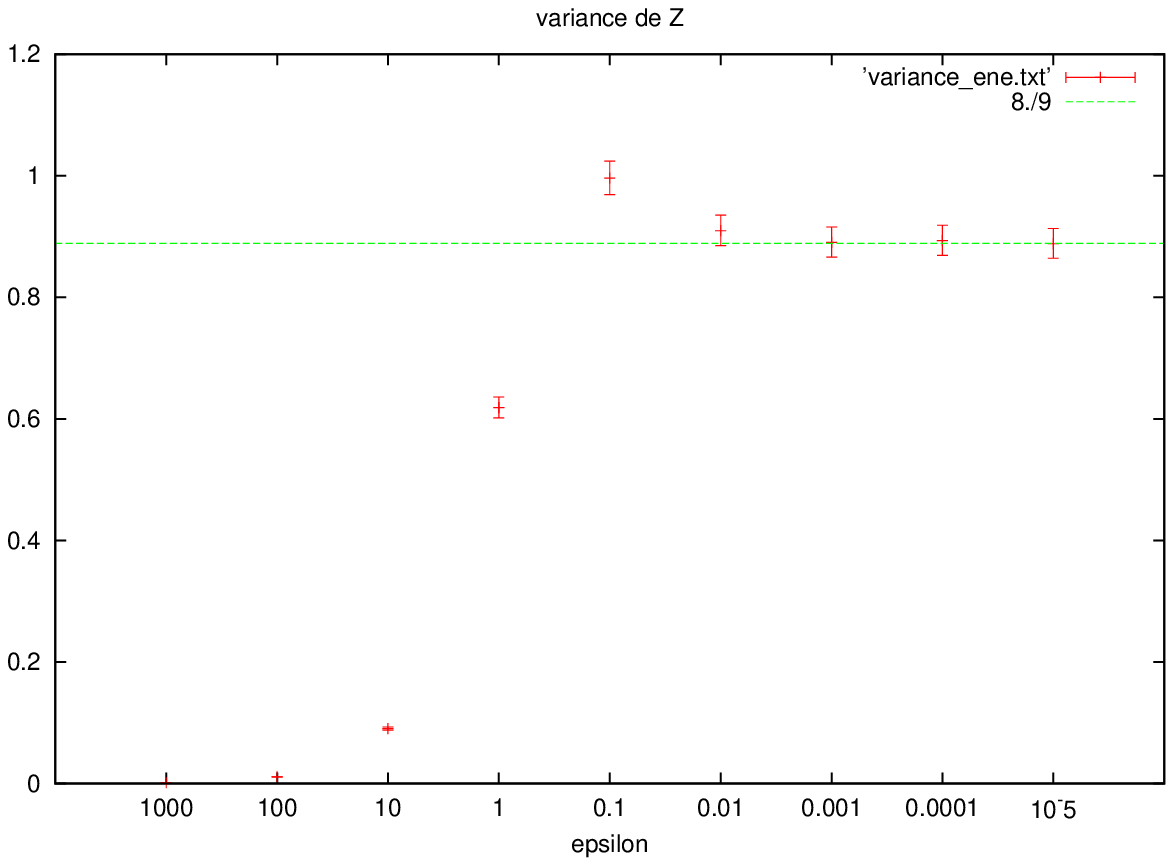}
\end{center}
\caption{Empirical expectation (left) and variance (right) of $\Delta
  Z^\epsilon$ as a function of $\epsilon$.
\label{Z_moyenne_variance_z}}
\end{figure}

\begin{figure}[h!]
\centerline{ 
\includegraphics[scale=0.3]{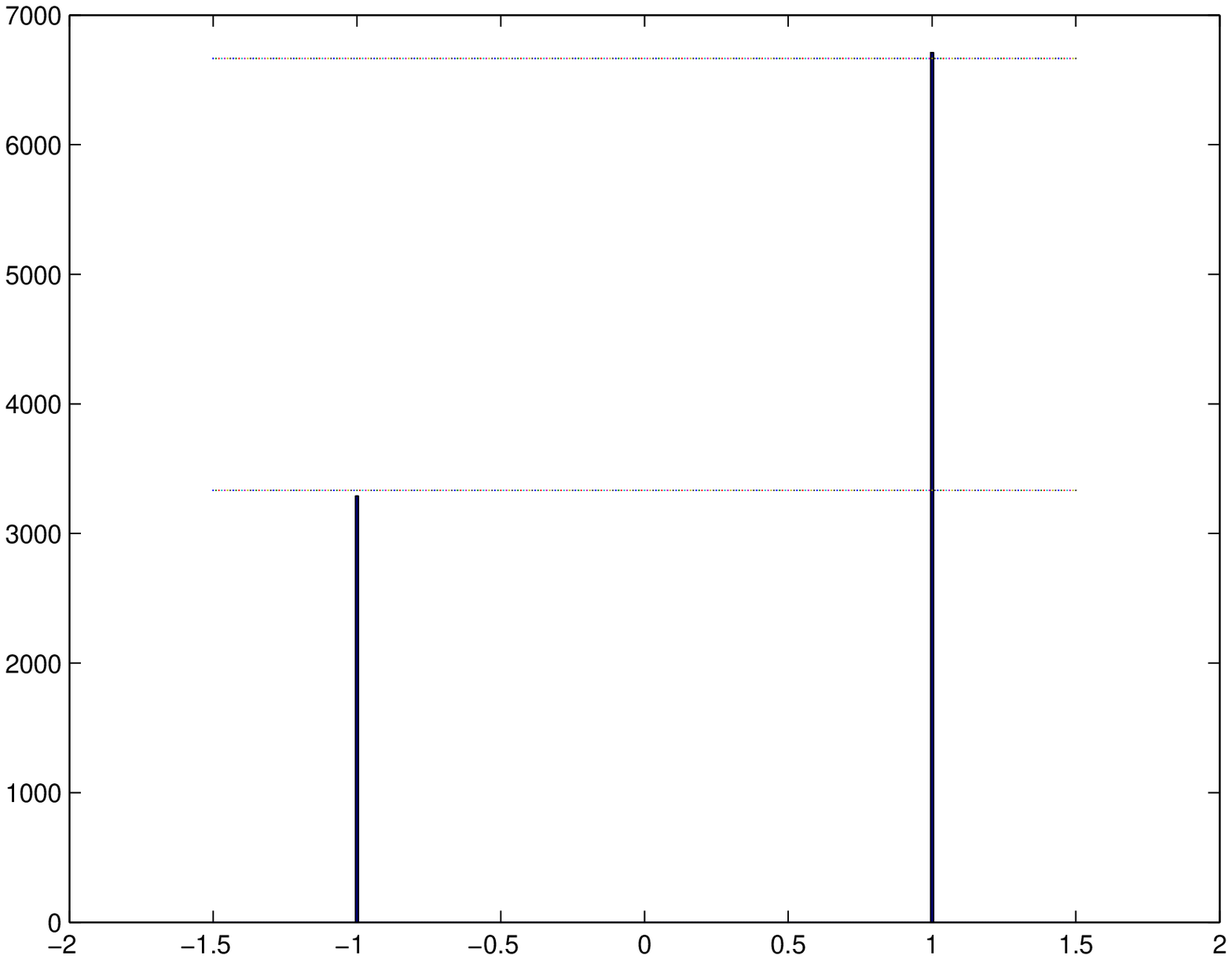}
\quad \quad
\includegraphics[scale=0.3]{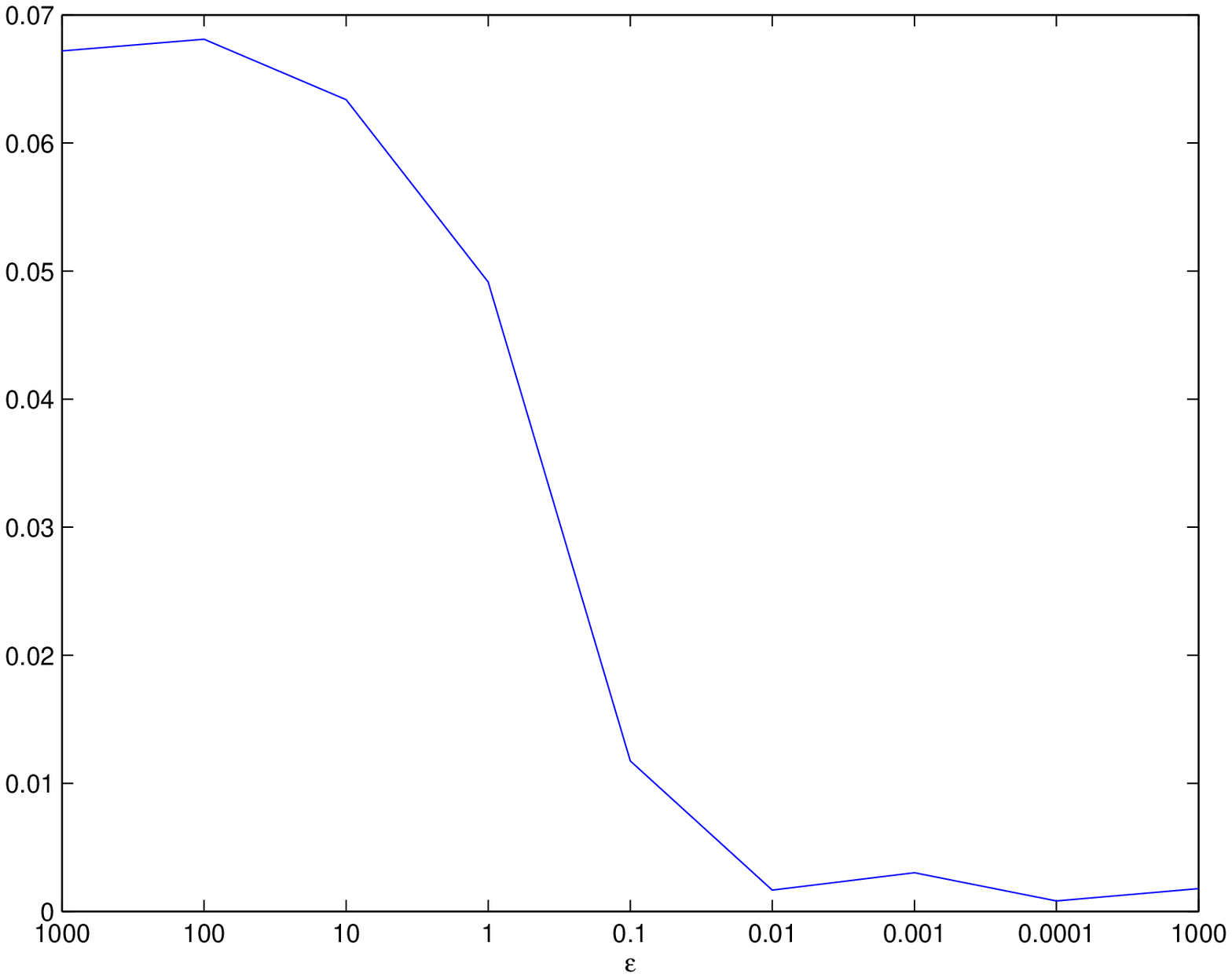}
}
\caption{Left: Empirical estimation of the probabilities $\mathbb P
  \left(\Delta Z^\epsilon =-1\right)$ and 
$\mathbb P \left(\Delta Z^\epsilon =1\right)$ for $\epsilon=10^{-5}$.
Right: $L^1$ error~\eqref{eq:err_L1} on the distribution of $\Delta
Z^\epsilon$ as a function of $\epsilon$.
\label{hist_Z_z}}
\end{figure}



\section{Exchange of energy in a system of two particles}
\label{sec_ene}

In this final section, we consider a more elaborate model. This model is
composed of two particles. The state of the first (resp. second) particle 
is described by the vector $X$ (resp. $Z$) with $k$ components. An energy
functional ${\cal E}$
is associated to each particle. The system evolves
either due to the internal evolution within a particle, or due to the interaction
between the two particles. In the first case, the energy of each particle 
is preserved. In the second case, the internal energy of each particle
varies, but the total energy of the system, 
${\cal E}(X) + {\cal E}(Z)$, is preserved. Interactions between the
particles occur on a much slower time scale than the internal evolution of
each particle. One must hence wait for a long time before observing any
change in each particle energy.

The model is presented in
details in Section~\ref{pres_ene}. In Section~\ref{section_3_2}, we
establish a convergence result on the time evolution of the energy of
the first particle, which is our macroscopic variable of interest. We
only give there a sketch of the proof as it follows the same arguments
as before. 

One of the interesting features of this model is that the macroscopic
variable of interest is {\em not} one cartesian coordinate of the
system. We show that the arguments used in
Sections~\ref{modele_deux_etats} and~\ref{infinite_detats} carry over to
this more general case. 

\subsection{Presentation of the model}
\label{pres_ene}

We consider a model with two particles. Each particle contains $k$
spin-like variables, that can take the value $0$ (spin down) or
$1$ (spin up). At time $t$, the state of the system is given by 
$\overline{Y}^\epsilon_t = 
\left(\overline{X}^\epsilon_t,\overline{Z}^\epsilon_t\right) 
\in M\times M$, where $M=\{0,1\}^k$ is the space for the $k$ spins of
each particle. For each particle, we are given an energy functional
$\mathcal E\left(x\right)=\mathcal E\left(x_1,\dots,x_k\right)$ (with
$x_j \in \{0,1\}$, $1 \leq j \leq k$) that
depends on the state of the $k$ spins of the particle. One choice is to
set $\mathcal E\left(x\right)=x_1+\dots+x_k$, which would correspond
(up to a multiplicative factor) to the energy of $k$ spins in a uniform
magnetic field. 

The intensity matrix of the process $\overline{Y}^\epsilon$ is built as
follows: 
\begin{itemize}
\item the internal dynamic of each particle is governed by an intensity
  matrix $Q$ that conserves its energy, i.e. $Q\left(x,x'\right)=0$ if
  $\mathcal E\left(x\right) \neq \mathcal E\left(x'\right)$. 
We define the global internal dynamic intensity matrix $\overline{Q}^0$ by 
\begin{eqnarray*}
\overline{Q}^0\left(\left(x,z\right),\left(x',z\right)\right)
&=&
Q\left(x,x'\right) \ \ \mbox{if $x\neq x'$},
\\
\overline{Q}^0\left(\left(x,z\right),\left(x,z'\right)\right)
&=&
Q\left(z,z'\right) \ \ \mbox{if $z\neq z'$},
\\
\overline{Q}^0\left(\left(x,z\right),\left(x',z'\right)\right)
&=&
0 \ \ \mbox{if $x\neq x'$ and $z \neq z'$}.
\end{eqnarray*}

\item the coupling between the two particles is described by a
  matrix $C$. This coupling introduces an exchange of energy between the
  two particles, while keeping the total energy constant. We assume that
  $C$ is such that
$$
C\left(\left(x,z\right),\left(x',z'\right)\right)= 0 
\mbox{ if $\mathcal E\left(x\right)+\mathcal E\left(z\right) \neq
\mathcal E\left(x'\right)+\mathcal E\left(z'\right)$ or if 
$\mathcal E\left(x\right) = \mathcal E\left(x'\right)$}.
$$

\item the transition intensities of the process $Y^\epsilon$ are given
  by 
$$
\overline{Q}^\epsilon=\overline{Q}^0+\epsilon C.
$$ 
\end{itemize}
We make the following assumption:
\begin{equation}
\label{ass:irred_e}
\begin{array}{c}
\text{the matrix $Q$ is such that, for every admissible energy level $e$,}
\\ \noalign{\vskip 3pt}
\text{the state class of energy $e$ is irreducible}
\\ \noalign{\vskip 3pt}
\text{and thus admits a unique invariant probability measure $\overline{\pi}^e$.}
\end{array}
\end{equation}
We denote by $\pi^e$ the probability measure on $M$ defined by
$\pi^e\left(x\right)=\overline{\pi}^e\left(x\right)$ if $\mathcal E
\left(x\right)=e$ and $\pi^e\left(x\right)=0$ otherwise. Any
normalized linear combination of the measures $\pi^e$ (with non-negative
coefficients) is thus an invariant
probability measure of $Q$. We consider the state classes of  
$M \times M$ such that the energy of each particle stays constant. These
classes are irreducible and admit a unique invariant probability measure
$\pi^e \otimes \pi^{e'}$. The invariant probability measures of
$\overline{Q}^0$ are of the form
$(Z')^{-1} \sum_{e,e'} Z\left(e,e'\right) \pi^e \otimes \pi^{e'}$, where
$Z\left(e,e'\right) \geq 0$ are some coefficients and where $Z'$ is a
normalization constant.

\subsection{Main result}
\label{section_3_2}

As pointed out above, our quantity of interest is 
$\mathcal E\left(\overline{X}^\epsilon_t\right)$, the energy of the
first particle. In view of the chosen scaling in
$\overline{Q}^\epsilon$, the characteristic time scale of evolution of
this energy is of the order of $\epsilon^{-1}$. We thus need to rescale
in time the evolution, and therefore introduce 
$Y^\epsilon_t
=
(X^\epsilon_t,Z^\epsilon_t)
:=\overline{Y}^\epsilon_{t/\epsilon}$ 
and 
$\mathcal E^\epsilon_t
=\mathcal E\left(\overline{X}^\epsilon_{t/\epsilon}\right)$. 

We now identify the limit of the process $\mathcal
E^\epsilon_t$, and state the main convergence result of that section,
namely Theorem~\ref{energie} below. 
Let
$L^\epsilon$ be the generator of $(Y^\epsilon_t)_{t\geq 0}$, which is a
jump process of intensity matrix $Q^\epsilon=\epsilon^{-1}
\overline{Q}^\epsilon$. We have
$$
L^\epsilon \phi\left(x,z\right)
=
\sum_{x',z'\in M}
Q^\epsilon\left(\left(x,z\right),\left(x',z'\right)\right)
\ [\phi\left(x',z'\right) - \phi\left(x,z\right)].
$$
For a function $\phi\left(x,z\right)=F\left(x\right)$ that only depends
on the state of the first particle, we have
\begin{eqnarray*}
\left(L^\epsilon F\right)\left(x,z\right) 
& = & 
\sum_{x',z'\in M} 
Q^\epsilon\left(\left(x,z\right),\left(x',z'\right)\right)
\ [F\left(x'\right) - F\left(x\right)]
\\
& = & 
\epsilon^{-1} \sum_{x'\in M} Q\left(x,x'\right) \ 
[F\left(x'\right) - F\left(x\right)]
\\
&& 
+
\sum_{x',z'\in M} C\left(\left(x,z\right),\left(x',z'\right)\right)
[F\left(x'\right) - F\left(x\right)].
\end{eqnarray*}
Now choosing $F=\mathcal E$, we obtain
$$
l\left(x,z\right)
:=
\left(L^\epsilon \mathcal E\right)\left(x,z\right) 
=
\sum_{x',z'\in M} C\left(\left(x,z\right),\left(x',z'\right)\right) \ 
[\mathcal E\left(x'\right) - \mathcal E\left(x\right)]
$$ 
since $Q\left(x,x'\right)=0$ if 
$\mathcal E\left(x'\right) \neq \mathcal E\left(x\right)$. We suppose
that, at the initial time, the energy of each particle is independent of
$\epsilon$: $\mathcal E(X^\epsilon_0)=E_x$ and $\mathcal E
\left(Z^\epsilon_0\right)=E_z$, where $E_x$ and $E_z$ are independent of
$\epsilon$. The total initial energy is denoted $E=E_x+E_z$. 

Using Proposition~\ref{lemma1}, we see that there exists a martingale
$M_t^\epsilon$ such that
\begin{equation}
\label{eq_gen}
\mathcal E^\epsilon_t=E_x+\int^t_0 l\left(X_s^\epsilon,
Z_s^\epsilon\right)ds +M_t^\epsilon.
\end{equation}
As in Section~\ref{section_1_2}, we can show that there exists a process 
$\mathcal E$ such that $\mathcal E^{\epsilon}$ converges to $\mathcal E$, up
to extraction. We now identify the distribution of the process $\mathcal
E$ and show that it is independent of the chosen sub-sequence (thereby
proving that all the sequence $\mathcal E^{\epsilon}$ converges to
$\mathcal E$, and not only a subsequence).

We introduce the average of the drift in~\eqref{eq_gen} with respect to
an invariant measure of $\overline{Q}^0$:
\begin{eqnarray*}
\widetilde{l}(e_1,e_2)
&=&
\sum_{x \text{ s.t. } \mathcal E(x)=e_1 \atop z \text{ s.t. } \mathcal
  E(z)=e_2} l(x,z) \pi^{e_1}(x) \pi^{e_2}(z)
\\
&=&
\sum_{x \text{ s.t. } \mathcal E(x)=e_1 \atop z \text{ s.t. } \mathcal
  E(z)=e_2} \pi^{e_1}(x) \pi^{e_2}(z)
\sum_{x',z'\in M} C\left(\left(x,z\right),\left(x',z'\right)\right) \ 
[\mathcal E\left(x'\right) - \mathcal E(x)]
\\
&=&
\sum_{x \text{ s.t. } \mathcal E(x)=e_1 \atop z \text{ s.t. } \mathcal
  E(z)=e_2} \pi^{e_1}(x) \pi^{e_2}(z)
\sum_{x',z' \text{ s.t. } \atop 
\mathcal E\left(x'\right)+ \mathcal E\left(z'\right)=e_1+e_2} 
C\left(\left(x,z\right),\left(x',z'\right)\right) \ 
[\mathcal E\left(x'\right) - e_1].
\end{eqnarray*}
We further define 
$$
f\left(x,z\right) = l(x,z) - 
\widetilde{l}\left(\mathcal E(x),\mathcal E(z)\right)
$$
and 
\begin{equation}
\label{eq:def_g_ene}
g(e)=\widetilde{l}\left(e,E-e\right),
\end{equation}
and recast~\eqref{eq_gen} as
\begin{equation}
\label{eq_gen_2}
\mathcal E^\epsilon_t = E_x + 
\int^t_0 f\left(X_s^\epsilon, Z_s^\epsilon\right)ds
+\int^t_0 g\left(\mathcal E_s^\epsilon\right)ds +M_t^\epsilon.
\end{equation}
We now want to pass to the limit $\epsilon \to 0$ in~\eqref{eq_gen_2}.

\medskip

Consider the second term in the right-hand side of\eqref{eq_gen_2}.
By construction, $f$ is the difference between the function $l$ and its
average $\widetilde{l}$. The average of $f$ is thus expected to
vanish. This is indeed the case: for any two energies $e_1$ and $e_2$,
we compute
\begin{eqnarray*}
\bra{\pi^{e_1} \otimes \pi^{e_2} }^T f
& = & 
\sum_{x \text{ s.t. } \mathcal E(x)=e_1 \atop z \text{ s.t. } \mathcal
  E(z)=e_2} \pi^{e_1}(x) \pi^{e_2}(z) f(x,z)
\\
&=&
\sum_{x \text{ s.t. } \mathcal E(x)=e_1 \atop z \text{ s.t. } \mathcal
  E(z)=e_2} \pi^{e_1}(x) \pi^{e_2}(z) l(x,z) 
- 
\sum_{x \text{ s.t. } \mathcal E(x)=e_1 \atop z \text{ s.t. } \mathcal
  E(z)=e_2} \pi^{e_1}(x) \pi^{e_2}(z) \widetilde{l}(e_1,e_2) 
\\
&=&
\sum_{x \text{ s.t. } \mathcal E(x)=e_1 \atop z \text{ s.t. } \mathcal
  E(z)=e_2} \pi^{e_1}(x) \pi^{e_2}(z) l(x,z) 
- 
\widetilde{l}(e_1,e_2) 
\\
& = & 0.
\end{eqnarray*}
Therefore, for any $\mu$ such that $\mu^T \overline{L}^0 =
0$, we have $\mu^Tf=0$. Following the arguments of
Proposition~\ref{prop1}, we deduce that, for any $t$, the random
variable $\dps \int^t_0 f\left(X_s^\epsilon, Z_s^\epsilon\right)ds$
converges to $0$ in $L^2(\O)$, and that the random process also weakly
converges to 0.

\medskip

We now turn to the third term of the right-hand side
of~\eqref{eq_gen_2}, and claim that (up to the extraction of a
sub-sequence) 
\begin{equation}
\label{conj}
\int^t_0
g\left(\mathcal E_s^\epsilon\right)ds 
\ \Rightarrow \ 
\int^t_0 g\left(\mathcal E_s \right)ds,
\end{equation}
where $\mathcal E_s$ is such that $\mathcal E_s^\epsilon  
\ \Rightarrow \ \mathcal E_s$.
The function $g$ is defined on the set $\mathcal E \left(M\right)$
of the
admissible energies, which is a finite set (we recall that 
$M=\{0,1\}^k$). We denote by $\widetilde{g}$ the
P1 interpolation of $g$ on $\mathbb R$, which is
a piecewise linear function defined on $\RR$ and that coincides with $g$
on $\mathcal E \left(M\right)$. The function $\widetilde{g}$ being
Lipschitz on $\mathbb R$, we infer from Lemma~\ref{lemme_continuite} that
the function $\dps \Phi : x \mapsto \left(\int^t_0 
\widetilde{g}\left(x_s\right)ds\right)_{t}$ is continuous. Therefore,
the convergence 
$\mathcal E_s^\epsilon \ \Rightarrow \ \mathcal E_s$ implies that
$$
\left( \int^t_0 g \left(\mathcal E_s^\epsilon\right)\right)
=
\left( \int^t_0 \widetilde{g} \left(\mathcal E_s^\epsilon\right)\right) 
\ \text{converges to} \ 
\left( \int^t_0 \widetilde{g}\left(\mathcal E_s \right)\right)
=
\left( \int^t_0 g\left(\mathcal E_s \right) \right).
$$ 
We thus have proved~\eqref{conj}.

\medskip

We next turn to the last term in the right-hand side of~\eqref{eq_gen_2}.
As in the previous sections, we can show that $M^\epsilon$ weakly
converges (up to extraction) to some martingale $M$. 

We can now pass to the limit $\epsilon \to 0$ in~\eqref{eq_gen_2}, and
obtain that the limit process $\mathcal E$ satisfies
\begin{equation}
\label{eq_gen_limite}
\mathcal E_t=E_x+\int^t_0 g \left(\mathcal E_s \right)ds +M_t.
\end{equation}
It is now easy to recast the above equation in a more useful form. 
In view of~\eqref{eq:def_g_ene}, we indeed note that
$$
g\left(e\right)
=
\sum_{e'}\underbrace{
\sum_{ x \text{ s.t. } \mathcal E(x)=e \atop z \text{ s.t. } \mathcal E(z)=E-e}
\sum_{ x' \text{ s.t. } \mathcal E\left(x'\right)=e' \atop 
z' \text{ s.t. } \mathcal E\left(z'\right)=E-e'} 
\pi^{e}(x) \pi^{E-e}(z) \ 
C\left(\left(x,z\right),\left(x',z'\right)\right)}_{B_E\left(e,e'\right)}
\ (e'-e).
$$
Therefore, the equation~\eqref{eq_gen_limite} reads
$$
\mathcal E_t
=
E_x + \int^t_0 \sum_{e'} B_E\left(\mathcal E_s,e'\right) \,
\left(e'-\mathcal E_s\right)ds +M_t,
$$
where, we recall, $E$ is the total energy of the system, which is
preserved along the dynamics.

\medskip

We conclude this formal approach by pointing out that the above equation
actually does not allow to identify the law of the process $(\mathcal
E_t)_t$. In the proof of Theorem~\ref{th1} (see
Section~\ref{section_1_2}), we performed that step of the proof by using
Lemma~\ref{lemme2}, which is not possible in our context here.
To identify the law of the process $(\mathcal E_t)_t$, we resort
to Lemma~\ref{appendixG}. Consider a bounded function $\phi$ on
$\mathcal E \left(M\right)$, and the martingale
$$
M_t^{\phi,\epsilon}
:=
\phi\left(\mathcal E^\epsilon_t\right)
- 
\phi\left(E_x\right) 
-\int^t_0 \sum_{e'} B_E\left(\mathcal E^\epsilon_s,e'\right) 
\left(\phi\left(e'\right) - \phi\left(\mathcal E^\epsilon_s\right)\right)ds.
$$
Following the same steps as above, we show that each term converges when
$\epsilon$ goes to zero. In particular, $M_t^{\phi,\epsilon}$ converges
to a martingale $M^\phi$ that satisfies
$$
M_t^\phi
=
\phi\left(\mathcal E_t\right)
- 
\phi\left(E_x\right) 
-\int^t_0 \sum_{e'} B_E\left(\mathcal E_s,e'\right) 
\left(\phi\left(e'\right) - \phi\left(\mathcal E_s\right)\right)ds.
$$
Lemma~\ref{appendixG} then implies that $\mathcal E$ is a
jump process of intensity matrix $B=B_E\left(e,e'\right)$. 

\medskip

We thus have the following result:

\begin{theorem}
\label{energie}
We denote by $\mathcal P^\epsilon$ the distribution of the process
$\left(\mathcal E^\epsilon \right)$, where we assumed that the initial
condition $(E_x,E_z)$ is independent of $\epsilon$. We denote by $\mathcal P$
the distribution of the jump process of initial condition $E_x$ and of
intensity matrix $B=B_E\left(e,e'\right)$, with $E=E_x+E_z$. Under the
assumptions on the matrices $Q$ and $C$ described in
Section~\ref{pres_ene}, we have
$$
\mathcal P^\epsilon \Rightarrow  \mathcal P 
\ \mbox{as $\epsilon \rightarrow 0$.}
$$
\end{theorem}

\subsection{Numerical illustration}
\label{sec_sim_ene}

We have numerically simulated the system described above, when each
particle has two spins, i.e. $k=2$. In this case, $\text{Card}(
M)=4$, and the admissible states for each particle are labelled as
$1$: $\downarrow
\downarrow$, $2$: $\uparrow
\downarrow$, $3$: $\downarrow
\uparrow$ and $4$: $\uparrow
\uparrow$.
The energy of each particle is the sum of the energies of its two spins,
which are equal to 0 (spin down, $\downarrow$) or 1 (spin up, $\uparrow$). 
The matrix $Q$ that governs the internal dynamic of each particle is of
the form 
$$
Q = \left(
\begin{array}{cccc}
0 & 0 &0 &0 \\
0 & -q_1& q_1 &0   \\
0 & q_2 & -q_2 &0 \\
0 &0 &0 &0  
\end{array}
\right).
$$
This matrix preserves the energy of the particle as it only allows
transitions between states of the same energy (namely, 
$\uparrow \downarrow$ and $\downarrow \uparrow$). We work with $q_1=10$
and $q_2=1$. 

\medskip

There are five possible initial energies for the complete system:
\begin{itemize}
\item $E=0$ (both particles are initially in the state $1$: $\downarrow
\downarrow$). The system then does not evolve, as only one
state corresponds to that total energy. The case when $E=4$ is similar.
\item $E=1$: initially, one particle is in the state $1$: $\downarrow
\downarrow$, while the other particle is in the state $2$: $\uparrow
\downarrow$ or $3$: $\downarrow \uparrow$. We consider this case 
below. Note that the case when $E=3$ is similar.
\item $E=2$: without loss of generality, we may assume that the initial
  state of each particle is $2$: $\uparrow \downarrow$.
\end{itemize}
In what follows, we only consider the case $E=1$. We have
checked that results obtained in the case $E=2$ lead to the same
qualitative conclusions. 

\medskip

As mentioned above, we assume that the initial state of the
first particle is $2$: $\uparrow \downarrow$ (corresponding to the
energy $E_x=1$), and that the initial state of the second
particle is $1$: $\downarrow \downarrow$ (corresponding to the
energy $E_z=0$).


The matrix $C$ (which encodes how the two particles interact) is chosen
of the form
\begin{eqnarray*}
C\left(\left(2,z\right),\left(x',z'\right)\right) &=& c_1 \ \mbox{if
  $\mathcal E(x)+\mathcal E(z) =
\mathcal E\left(x'\right)+\mathcal E\left(z'\right)$ and
$\mathcal E(x) \neq \mathcal E\left(x'\right)$},
\\
C\left(\left(x,z\right),\left(x',z'\right)\right) &=& c_2 \ \mbox{if
  $x\neq 2$ and $\mathcal E(x)+\mathcal E(z) =
\mathcal E\left(x'\right)+\mathcal E\left(z'\right)$
and $\mathcal E(x) \neq \mathcal E\left(x'\right)$},
\\
C\left(\left(x,z\right),\left(x',z'\right)\right) &=& 0
\ \mbox{otherwise}.
\end{eqnarray*}
We work with $c_1=1$ and $c_2=0.2$.

\medskip

We monitor the distribution of $S^\epsilon_0$, the first waiting time
before an exchange of energy between the two
particles occurs. Figures~\ref{moyenne_variance_ene} and~\ref{hist_ene}
show the convergence of the distribution of
$S^\epsilon_0$ to the asymptotic distribution, which is an exponential
distribution of parameter $B\left(1,0\right) = 6/11$.

\begin{figure}[h]
\begin{center}
\psfrag{'moyenne_ene_gnu.txt'}{\tiny \hspace{0.3cm} mean}
\psfrag{'variance_ene_gnu.txt'}{\tiny variance}
\includegraphics[scale=0.4]{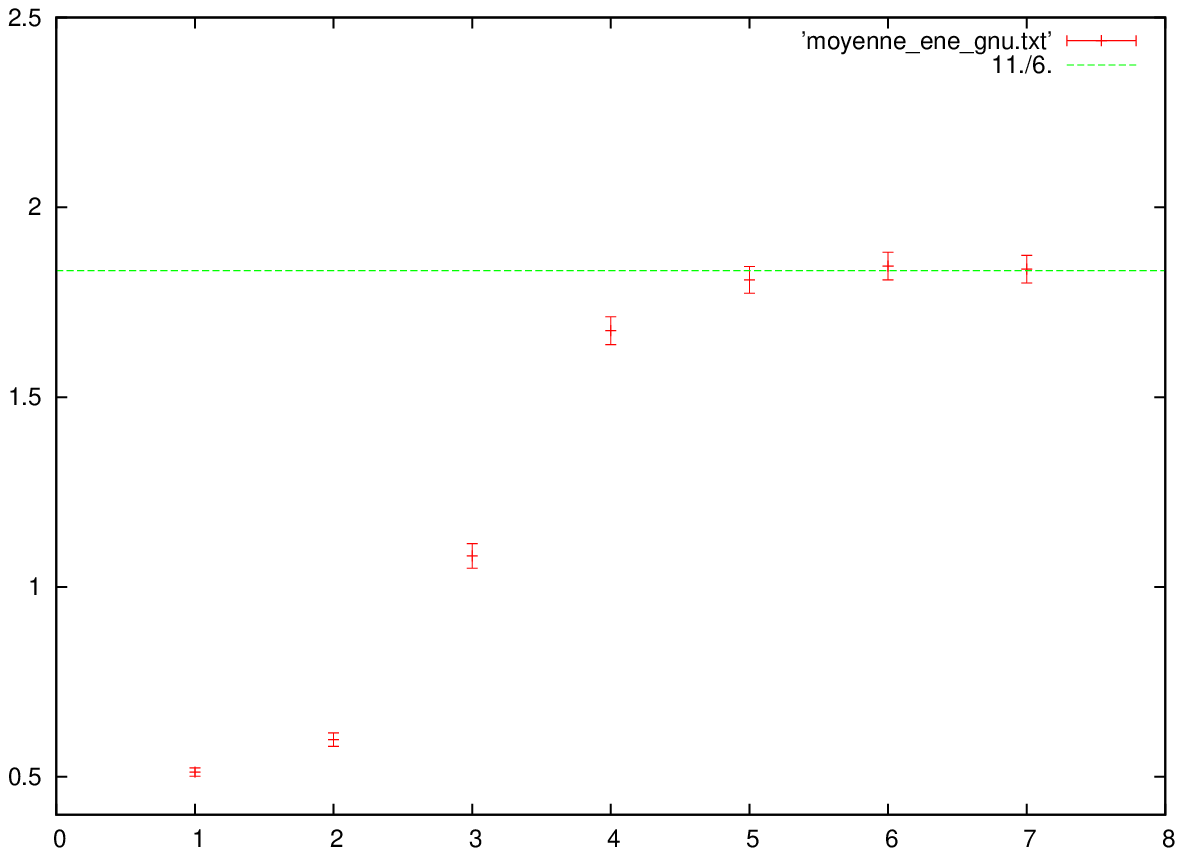}
\includegraphics[scale=0.4]{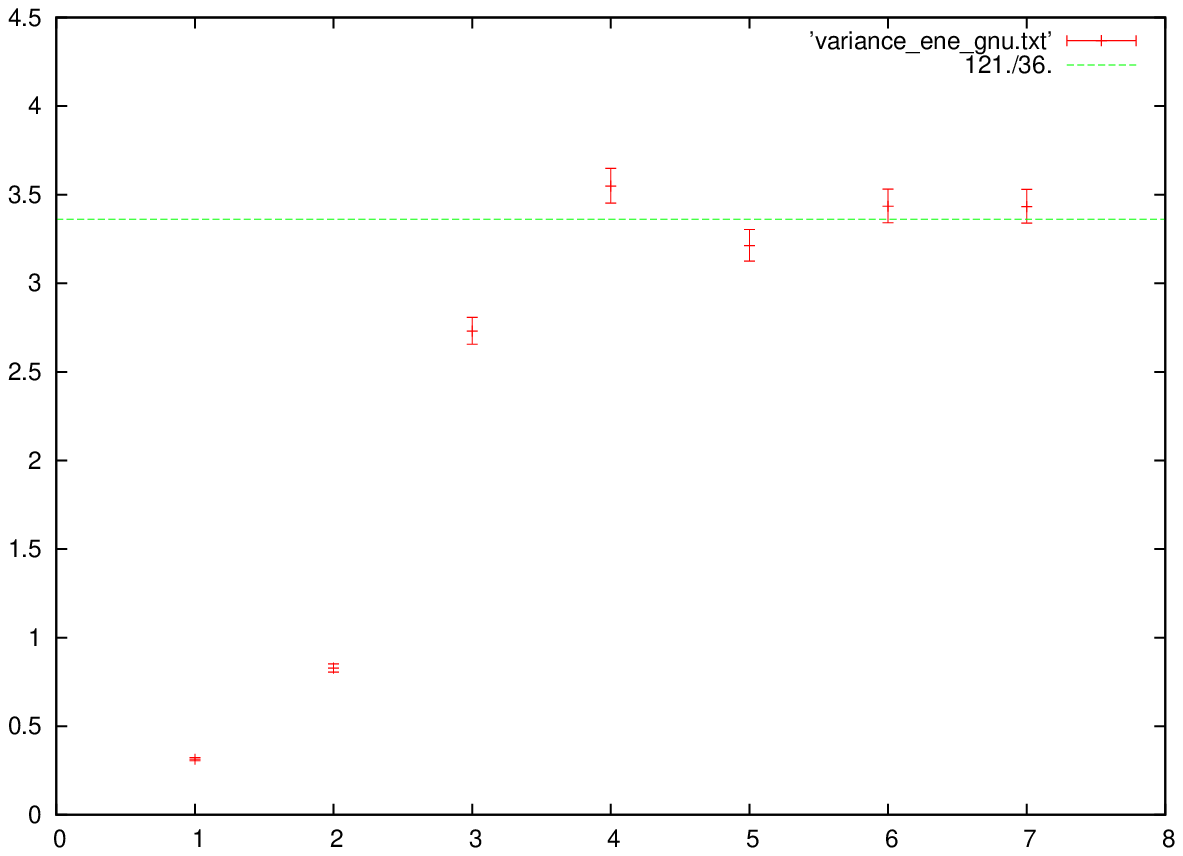}
\end{center}
\caption{Empirical expectation (left) and variance (right) of
  $S^\epsilon_0$ as a function of $\epsilon$.
\label{moyenne_variance_ene}}
\end{figure}


\begin{figure}[h!]
\centerline{ 
\includegraphics[scale=0.35]{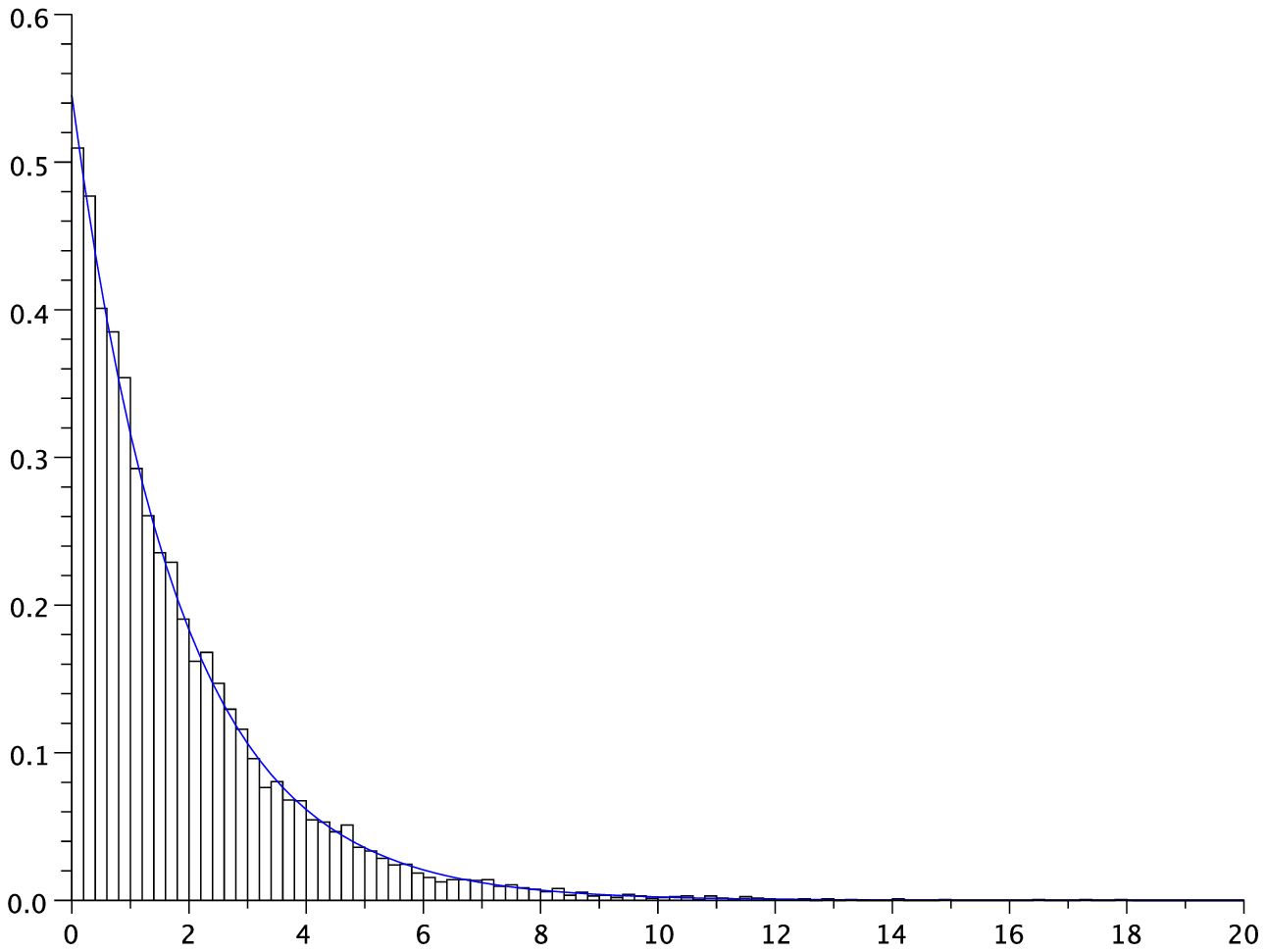}
\includegraphics[scale=0.25]{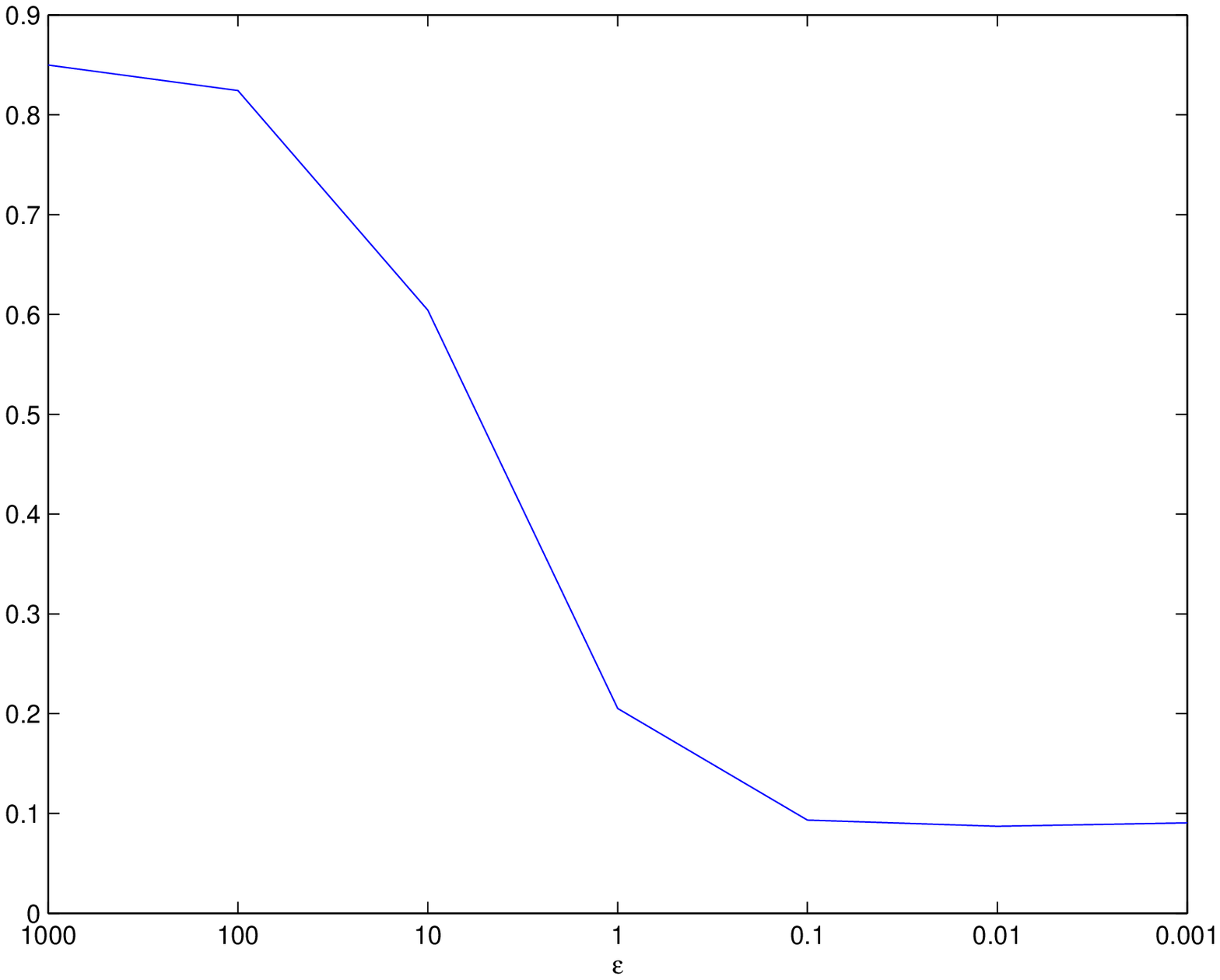}
}
\caption{Left: Distribution of the first waiting time $S^\epsilon_0$ before
  the energy of the first particle changes
  ($\epsilon=10^{-3}$);
Right: $L^1$ error~\eqref{eq:err_L1} between the distribution of
  $S^\epsilon_0$ and its limit distribution.
\label{hist_ene}}
\end{figure}


\section*{Acknowledgments}

The models we consider here and the questions we study were suggested to
us by Stefano Olla. We thank him for his suggestions and the fruitful
discussions we had with him. We also thank Eric Canc\`es for his
constant support throughout this research 
project. We thank Tony Leli\`evre for his careful reading of a previous
version of this article. The research leading to these results has
received funding 
from the European Research Council under the European Community's
Seventh Framework Programme (FP7/2007-2013 Grant Agreement MNIQS no.
258023).

\appendix

\section{Some useful results}
\label{sec:app}

For convenience, we recall in this Appendix some classical results of
probability theory that are needed in this article.



\paragraph{Martingales}




Several results on martingales are useful in this work. 
The first one is an existence and uniqueness result for the martingale
problem introduced by D.W. Stroock and S.R.S. Varadhan (see
e.g.~\cite{bass} and~\cite{SV}): 

\begin{proposition}[Lemma 5.1 of Appendix 1 of~\cite{kipnis}]
\label{lemma1}
Let $\left(X_t\right)_{t \geq 0} $ be a Markov process and let
$\left(\mathcal F_t\right)_{t \geq 0}$ be its natural filtration. For
any bounded function $F$, we introduce
$$
M^F_t = F\left(X_t\right) - F\left(X_0\right) - \int^t_0 LF\left(X_s\right)ds
$$ 
and 
$$
N^F_t = \left(M^F_t\right)^2 - \int^t_0 \left(LF^2\left(X_s\right) -
  2F\left(X_s\right) LF\left(X_s\right)\right) ds,
$$
where $L$ is the generator of the Markov process $(X_t)$. Then  $M^F$
and $N^F$ are $\mathcal F_t$-martingales. In particular, the quadratic
variation of $M^F$ reads
$$
\langle M^F\rangle _t
=
\int^t_0 \left(LF^2\left(X_s\right) - 2F\left(X_s\right)
  LF\left(X_s\right) \right)ds.
$$
\end{proposition}
We recall that for a continuous local martingale $M$, the process
$\langle M \rangle$ is defined to be the unique right-continuous and
increasing predictable process starting at zero such that $M^2-\langle
M\rangle$ is a local martingale. 

\medskip

The next result is of paramount importance to prove that a process is a
jump process, and to identify its generator. We state here this result
as a simplified version of~\cite[Theorem 21.11]{kallenberg}.

\begin{lemma}[Uniqueness result for the martingale problem]
\label{appendixG} 

Let $F$ be a countable space, $Z_t$ a stochastic process valued in $F$
and $L$ an operator on bounded functions $\phi:F\rightarrow \RR$ defined by 
$$
L\phi (x)=\sum_{x'\in F}L_{x,x'}\bra{\phi(x')-\phi(x)},
$$
where $L_{x,x'}\geq 0$ for any $x, x'\in F$ and $\sup_{x,x'\in F}L_{x,x'}<\infty$.
If for any bounded function $\phi:F\rightarrow \RR$, the process
$$
M_t^\phi := \phi\left(Z_t\right) - \phi\left(Z_0\right) -\int_0^t
L\phi\left(Z_s\right) \, ds
$$ 
is a martingale w.r.t. the natural filtration of $(Z_t)_{t\geq 0}$, then
$(Z_t)_{t\geq 0}$ is the jump process of initial
condition $Z_0$ and of generator $L$.
\end{lemma}

\paragraph{Convergence of probability measures}

We now turn to classical results concerning the convergence of
probability measures in $D_{\mathbb R}\left[0,\infty\right)$, which is
the space of functions that are right continuous with left limits (the
so-called c\`ad-l\`ag functions), defined on $\left[0,\infty\right)$ and
valued in $\mathbb R$. Proposition~\ref{convergence_D} gives an
equivalent definition of the Skorohod metric on $D_{\mathbb
  R}\left[0,\infty\right)$ (see~\cite[p.~116-118]{kurtz} for the
original definition of the Skorohod metric, that we actually do not use
in this work).
Theorem~\ref{critere} 
states convergence criteria for probability measures on $D_{\mathbb
  R}\left[0,\infty\right)$. 

\begin{proposition}[Proposition 5.3, Chap. 3 of~\cite{kurtz}]
\label{convergence_D}  

Let $(x_n)_{n\geq 0}$ be a sequence in $D_{\mathbb
  R}\left[0,\infty\right)$ and $x\in D_{\mathbb R}\left[0,\infty\right)$.
The following assertions are equivalent:
\begin{itemize}
\item $\dps \lim_{n \to \infty} x_n = x$ in the space 
$D_{\mathbb R}\left[0,\infty\right)$ endowed with the Skorohod metric.

\item For any $T > 0$, there exists a sequence of strictly increasing,
  continuous maps $(\lambda_n)_{n\geq 0}$ defined on
  $\left[0,\infty\right)$ and valued in $\left[0,\infty\right)$ such that 
\begin{equation}
\label{eq:convergence_D1} 
\lim_{n \rightarrow \infty} \sup_{0\leq t \leq T} |\lambda_n\left(t\right)-t|=0
\end{equation}
and
\begin{equation}
\label{eq:convergence_D2} 
\lim_{n\rightarrow \infty} \sup_{0\leq t \leq T} 
|x_n\left(t\right)-x\left(\lambda_n\left(t\right)\right)|=0.
\end{equation}
\end{itemize}
\end{proposition}

\begin{theorem}[Aldous' criterion, Theorem VI.4.5 of~\cite{jacod}]
\label{critere} 
Let $(X^n)_{n\geq 1}$ be a sequence of c\`ad-l\`ag processes, with distributions
$\P^n$. Suppose that
\begin{itemize}
 \item for any $N\in \NN$ and $\epsilon>0$, there exists
   $n_0\in\NN$, $n_0>0$, and $K\in\RR^+$ such that, for any $n\geq n_0$,
\begin{equation}
\label{3.21i}
\P^n \left( \sup_{t\leq N} \left|X^n_t\right|>K \right) \leq \epsilon.
\end{equation}
\item for any $N\in \NN$ and $\alpha>0$, we have
\begin{equation}
\label{4.4}
\lim_{\theta\rightarrow 0} \, \limsup_n \, 
\sup_{S,T\in\mathfrak T_N^n, \;S\leq T\leq S+\theta}
\P^n\left(\left|X^n_T-X^n_S\right| \geq \alpha\right)=0,
\end{equation}
where $\mathfrak T_N^n$ is the set of all $\F^n$ stopping times that are
bounded by $N$.
\end{itemize}
Then the sequence $(X^n)_{n\in\NN}$ is tight.
\end{theorem}

\bibliographystyle{plain}
\bibliography{fichierb} 

\begin{thebibliography}{10}

\bibitem{bass}
R.~F. Bass.
\newblock Uniqueness in law for pure jump {M}arkov processes.
\newblock {\em Probab. Theory Related Fields}, 79(2):271--287, 1988.

\bibitem{billy}
P.~Billingsley.
\newblock {\em Convergence of probability measures}.
\newblock John Wiley \& Sons Inc., New York, 1968.

\bibitem{billy2}
P.~Billingsley.
\newblock {\em Convergence of probability measures}.
\newblock Wiley Series in Probability and Statistics: Probability and
  Statistics. John Wiley \& Sons Inc., New York, second edition, 1999.

\bibitem{kurtz}
S.N. Ethier and T.G. Kurtz.
\newblock {\em Markov processes}.
\newblock Wiley Series in Probability and Mathematical Statistics: Probability
  and Mathematical Statistics. John Wiley \& Sons Inc., New York, 1986.

\bibitem{friesecke}
G.~Friesecke, O.~Junge, and P.~Koltai.
\newblock Mean field approximation in conformation dynamics.
\newblock {\em SIAM Mult. Mod. Sim.}, 8(1):254--268, 2009.

\bibitem{gks}
D.~Givon, R.~Kupferman, and A.M. Stuart.
\newblock Extracting macroscopic dynamics: model problems and algorithms.
\newblock {\em Nonlinearity}, 17(6):55--127, 2004.

\bibitem{jacod}
J.~Jacod and A.N. Shiryaev.
\newblock {\em Limit theorems for stochastic processes}, volume 288 of {\em
  Grundlehren der Mathematischen Wissenschaften}.
\newblock Springer-Verlag, Berlin, second edition, 2003.

\bibitem{kallenberg}
O.~Kallenberg.
\newblock {\em Foundations of modern probability}.
\newblock Probability and its Applications. Springer-Verlag, New York, second
  edition, 2002.

\bibitem{kipnis}
C.~Kipnis and C.~Landim.
\newblock {\em Scaling limits of interacting particle systems}, volume 320 of
  {\em Grundlehren der Mathematischen Wissenschaften}.
\newblock Springer-Verlag, Berlin, 1999.

\bibitem{these_salma}
S.~Lahbabi.
\newblock PhD thesis, Ecole Nationale des Ponts et Chaussées, 2013.
\newblock in preparation.

\bibitem{par_rep}
C.~{L}e Bris, T.~Leli\`evre, M.~Luskin, and D.~Perez.
\newblock A mathematical formalization of the parallel replica dynamics.
\newblock {\em Monte Carlo Methods and Applications}, 18(2):119--146, 2012.

\bibitem{eff_dyn}
F.~Legoll and T.~Leli\`evre.
\newblock Effective dynamics using conditional expectations.
\newblock {\em Nonlinearity}, 23(9):2131--2163, 2010.

\bibitem{eff_dyn_lncse}
F.~Legoll and T.~Leli\`evre.
\newblock Some remarks on free energy and coarse-graining.
\newblock In B.~Engquist, O.~Runborg, and R.~Tsai, editors, {\em Numerical
  Analysis and Multiscale Computations}, volume~82 of {\em Lecture Notes in
  Computational Sciences and Engineering}, pages 279--329. Springer, 2012.

\bibitem{novotny}
M.~A. Novotny.
\newblock Monte {C}arlo algorithms with absorbing {M}arkov chains: Fast local
  algorithms for slow dynamics.
\newblock {\em Phys. Rev. Lett.}, 74(1):1--5, 1995.

\bibitem{pav_stu}
G.A. Pavliotis and A.M. Stuart.
\newblock {\em Multiscale methods: averaging and homogenization}.
\newblock Springer, 2007.

\bibitem{revuz}
D.~Revuz and M.~Yor.
\newblock {\em Continuous martingales and {B}rownian motion}, volume 293 of
  {\em Grundlehren der Mathematischen Wissenschaften}.
\newblock Springer-Verlag, Berlin, third edition, 1999.

\bibitem{schuette_jcp}
C.~Sch\"utte, A.~Fischer, W.~Huisinga, and P.~Deuflhard.
\newblock A direct approach to conformational dynamics based on {H}ybrid
  {M}onte-{C}arlo.
\newblock {\em J. Comp. Phys.}, 151:146--168, 1999.

\bibitem{schuette_handbook}
C.~Sch\"utte and W.~Huisinga.
\newblock Biomolecular conformations can be identified as metastable sets of
  molecular dynamics.
\newblock In P.G. Ciarlet and C.~Le Bris, editors, {\em Handbook of Numerical
  Analysis (Special volume on computational chemistry)}, volume~X, pages
  699--744. Elsevier, 2003.

\bibitem{SV}
D.W. Stroock and S.R.S. Varadhan.
\newblock {\em Multidimensional diffusion processes}, volume 233 of {\em
  Grundlehren der Mathematischen Wissenschaften}.
\newblock Springer-Verlag, Berlin, 1979.

\end{thebibliography}

\end{document}